\documentclass[a4paper]{article}
\usepackage{geometry}
\usepackage{amsmath}
\usepackage{amsthm}
\usepackage{amssymb}
\usepackage{hyperref}
\usepackage{xcolor}
\usepackage{comment}
\usepackage{graphicx} 
\usepackage{cite}

\usepackage{algorithm}
\usepackage{algpseudocode}

\usepackage{matlab-prettifier} 

\usepackage[normalem]{ulem}

\usepackage{tikz}
\usepackage{pgfplots}

\newcommand{\vect}[1]{\mathrm{vec}(#1)}

\newcommand{\br}{\color{black}}
\newcommand{\er}{\color{black}}

\newtheorem{theorem}{Theorem}[section]
\newtheorem{lemma}[theorem]{Lemma}
\newtheorem{corollary}[theorem]{Corollary}

\theoremstyle{definition}
\newtheorem{definition}[theorem]{Definition}

\theoremstyle{remark}
\newtheorem{remark}[theorem]{Remark}

\title{Backward errors for multiple eigenpairs in structured and unstructured nonlinear eigenvalue problems}
  
\author{
    Miryam Gnazzo\thanks{Dipartimento di Matematica, Università di Pisa, Italy (\href{mailto:miryam.gnazzo@dm.unipi.it}{\tt miryam.gnazzo@dm.unipi.it}).},\ \ 
    Leonardo Robol\thanks{Dipartimento di Matematica, Università di Pisa, Italy (\href{mailto:leonardo.robol@unipi.it}{\tt leonardo.robol@unipi.it}).}
}
\date{\today}

\begin{document}

\maketitle

\begin{abstract}
    Given a nonlinear matrix-valued function $F(\lambda)$ and 
    approximate eigenpairs $(\lambda_i, v_i)$, 
    we discuss how  
    to determine the smallest perturbation
    $\delta F$ such that $[F + \delta F](\lambda_i) v_i = 0$; 
    we call the distance between the $F$ and $F + \delta F$
    the \emph{backward error} for this 
    set of approximate eigenpairs. We focus on the case where 
    $F(\lambda)$ is given as a linear combination of scalar functions
    multiplying matrix coefficients $F_i$, and the perturbation 
    is done on the matrix coefficients. We provide inexpensive upper 
    bounds, and a way to accurately compute the backward error
    by means of direct computations or through
    Riemannian optimization. 
    We also discuss how the backward error can be determined when 
    the $F_i$ have particular structures (such as symmetry, sparsity, or low-rank), and the perturbations are required to preserve them. For special cases (such as for symmetric coefficients), 
    explicit and inexpensive formulas to compute the 
    $\delta F_i$ are also given. 
\end{abstract}


\textbf{MSC Class}:  65H17, 65F15, 15A18.

\section{Introduction}
We consider matrix-valued functions $F: \mathbb{C} \mapsto \mathbb{C}^{n \times n}$ and the related 
nonlinear eigenvalue problem, that consists in finding $\lambda \in \mathbb{C}$ and $v$ such that 
\[
F(\lambda) v=0, \qquad 
v \in \mathbb C^n \setminus \{ 0 \}.
\]
The pair $(\lambda, v)$ is called \emph{eigenpair}, and 
$\lambda$, $v$ are an \emph{eigenvalue} 
and an \emph{eigenvector}
for $F(\lambda)$, respectively. Quite often, 
the matrix-valued function $F(\lambda)$ is given in 
\emph{split form} as a linear combination of matrix 
coefficients multiplied by \br scalar functions:\er 
\begin{equation} 
\label{eq:split form}
 F(\lambda) = f_1(\lambda) F_1 + \ldots + f_k(\lambda) F_k.
\end{equation}
The coefficients $F_j$ frequently encode data coming from the 
underlying application (for instance, the coefficients of 
a stiffness or damping matrix in the PDE setting). A few representative cases of this application can be found in \cite{NLEVP}; for instance see the quadratic eigenvalue problem \texttt{spring} associated with a damped mass-spring system \cite[Example 2]{Tisseur}. It is often 
useful to consider non-polynomial scalar functions $f_j$, such as exponentials. This 
kind of matrix-valued functions arise in the context of constant-coefficients delay differential equations, see for instance \cite{Corless}. We will only consider nonlinear eigenvalue problems of the form \eqref{eq:split form}
in this work. 

Nonlinear eigenvalue problems arise in a wide set of applications; a large collection of examples can be found in the MATLAB package \texttt{nlevp} \cite{NLEVP}. Several algorithms have been developed for the numerical solution of nonlinear eigenvalue problems, see for instance \cite{GutTiss} for a survey on the nonlinear eigenvalue problem. An implementation of possible solvers for this problem is available in the Julia package NEP-PACK \cite{neppack}. Moreover, recent solvers employ a variant of the AAA algorithm for the solution of the nonlinear eigenvalue problem \cite{LietaertMeer}. 

Selected instances of this problem have been thoroughly studied 
in the literature: when $f_j(\lambda) = \lambda^{j-1}$ we obtain a 
\emph{polynomial eigenvalue problem}. For instance a complete review on the quadratic eigenvalue problem can be found in \cite{TissMeer}. In the case $k = 2$ we 
get a matrix pencil, or a standard eigenvalue problem. 

In numerical linear algebra, the standard way to assess the 
quality of computed eigenvalues and eigenvectors is to determine 
their backward error; the latter is defined as the distance from 
the closest eigenvalue problem for which the computed eigenvalues 
(and eigenvectors) are exact. 

For standard eigenvalue problems, we have explicit formulas that relate the residual norm $\| F(\lambda) v \|_2$ to the backward error, and that are part of any numerical linear algebra textbook; similar results can be given for polynomial eigenvalue problems \cite{Tisseur}. Moreover in \cite{DopicoKron} the authors proposed a backward error analysis for the solution of the polynomial eigenvalue problems and complete polynomial eigenproblems, via block Kronecker linearizations. Some results can be found for 
more general nonlinear eigenvalue problems as well. In \cite{AhmMehr}, a formulation for the backward error of a given eigenpair has been proposed in the context of homogeneous nonlinear eigenvalue problems, with particular attention to structured matrix-valued functions $F$. For rational eigenvalue problems, in \cite{Sharma} the authors derive formulae for the symmetric backward error of one eigenvalue. In \cite{KarKreMeng} the authors develop a characterization for the backward error associated with a set of eigenvalues for a matrix-valued analytic function. Their bound however does not relate with the split form of \eqref{eq:split form}, 
but rather focus on finding a small functional perturbation. 

The contribution of this work is twofold:
\begin{enumerate}
    \item We provide computable and inexpensive
      bounds for the backward error of a \emph{set of 
      eigenvalues and eigenvectors} (or for just 
      the eigenvalues, if no eigenvectors have been computed). 
    \item We give numerical procedures based on 
      Riemannian optimization that compute 
      the backward error accurately, and that do so
      retaining any structure found in the coefficients 
      $F_j$ (such as sparsity, low-rank, symmetries, \ldots).
\end{enumerate}
We also provide computable bounds for the backward 
error in the structured case, 
but these will be inexpensive only for the case of 
symmetric nonlinear eigenvalue problems. For more general structures, we found that directly computing the backward error 
is often the best way to proceed. 

A previous attempt at characterizing structured backward errors 
for nonlinear eigenvalue problems 
can be found in \cite{AhmMehr}, where the authors focus on only 
one eigenpair. Our bounds will reduce to the one in this work 
when considering a set with a single eigenpair. 

We remark that it may be \br tempting \er to use results for a single 
eigenpair to draw conclusions on the accuracy of a set 
of eigenpairs, but this can be misleading. Indeed, it may happen that the backward errors of two different approximate eigenvalues $\lambda_1$, $\lambda_2$ are small, but there is no close-by 
nonlinear eigenvalue problem that has both as eigenvalues. 
An example showcasing this possibility may be found in Section $6.2$ in \cite{KarKreMeng}. 

A similar discussion of considering a set of eigenvalues at once 
for standard eigenvalue problems can be found in \cite{TisseurChart}, together with a complete survey on structured normwise backward errors for a set of eigenpairs.
Following \cite{TisseurChart}, we define the backward error 
associated with a set of $p$ eigenpairs $(\hat \lambda_i, \hat v_i)$ 
as follows:
\[
\eta := \min \left\lbrace \|\left[ \delta F_1, \ldots, \delta F_k \right] \|_F \  \Big|\ \sum_{j=1}^k f_j(\hat \lambda_i) (F_j + \delta F_j) \hat v_i =0, \; 1 \leq i \leq p \right\rbrace.
\]
This measure does not take into account possible additional structures on the coefficient matrices $F_j$. Nevertheless it may be useful to include structures into account, for instance in situations where we would like to exploit the structure of the problem and therefore the sensitivity of the solution of the nonlinear eigenvalue problem should be measured with respect to the same structure. This is done for instance in \cite{DopicoPomes} where the authors propose a computable structured condition number for the class of parametrized quasiseparable matrices. Results 
on condition numbers of structured matrix polynomials, such symmetric and palindromic,  have been developed by Adhikari et al. in \cite{Alam}. In addition, the authors analyze the relations with the condition number of structured linearizations of the matrix polynomials.
For all contexts where 
the coefficient belong to prescribed classes of structured 
matrices, we will also introduce a structured 
backward error $\eta_{\mathcal S}$ defined analogously, but with the 
constraint of $F_j + \delta F_j$ sharing the same structure 
of $F_j$. We will discuss in detail the cases of an assigned 
sparsity pattern, a maximum rank, and symmetries. The proposed 
algorithm will be able to deal with different structures for each coefficient, and also multiple structures at once with no
modifications.

The paper is organized as follows. In Section \ref{sec:unstructured}, we analyze the unstructured backward error for the nonlinear eigenvalue problem and provide several computable upper bounds for it. In Section \ref{sec:structured}, we present an overview of structured backward errors for a set of approximated eigenpairs. For the case of general linear structures, we provide a formula for the backward error. Then we specialize the results for the symmetry structure, providing a cheaper computable upper bound for the backward error. In the end, we consider nonlinear structures, \br such as fixed rank matrices\er, for which we are able to provide an upper bound computed through the use of a Riemannian optimization-based technique. In Section \ref{sec:numerical experiments} a few numerical tests and examples conclude the paper.

\section{Backward errors for nonlinear eigenvalue problems}
\label{sec:unstructured}

We consider matrix-valued functions $F:\mathbb{C} \mapsto \mathbb{C}^{n \times n}$ in split form \eqref{eq:split form}, 
that is {\small{$F(\lambda):= \sum_{j=1}^k{F_j} f_j(\lambda)$}}, where
$F_j \in \mathbb{C}^{n\times n}$ and $f_j: \mathbb{C} \mapsto \mathbb{C}$ 
\br scalar functions \er for $j=1,\ldots,k$. Observe that given a general matrix-valued function, it is always possible to write it in split form, decomposing it as $\left[ F(\lambda)\right]_{ij}e_ie_j^T$, for $i,j=1,\ldots,n$, where $e_i,e_j$ are vectors of the canonical basis, 
so this is not restrictive. 
On the other hand, nonlinear eigenvalue problems 
arising in applications are often naturally given 
in this form with  a small $k$ \cite{NLEVP,neppack}. In 
particular, the formulation \eqref{eq:split form} also includes matrix polynomials of degree $k-1$.

\subsection{Backward errors for given eigenpairs}
Consider the nonlinear eigenvalue problem
$F(\lambda) v = 0$, and 
assume that we have identified $p$ approximate 
eigenpairs, for which we have the relations 
\begin{equation}
\label{eq:residuals}
  F(\hat \lambda_i) \hat v_i = \sum_{j = 1}^k f_j(\hat \lambda_i) F_j \hat v_i = r_i.     
\end{equation}
The vectors $r_i$ are the residuals. We provide 
the formal definition of backward errors for these 
approximate eigenpairs. 
\begin{definition} \label{def:unstructured-be}
    Given a nonlinear matrix-valued function $F(\lambda)$, consider $p$ approximate eigenpairs $(\hat \lambda_i, \hat v_i)$, for $i=1,\ldots,p$. We define the \emph{backward error of the eigenpairs} $(\hat \lambda_i, \hat v_i)$ as
    \[
        \eta := \min \Bigg\{ 
        \left\|\left[ \delta F_1,\ldots, \delta F_k \right] \right\|_F \ | \ 
        \sum_{j = 1}^k f_j(\hat \lambda_i) (F_j + \delta F_j) \hat v_i = 0 , 
        1 \leq i \leq p \Bigg\}.
    \]
\end{definition}

\begin{remark} \label{rem:solvability}
We note that a trivial solution always exists by taking $\delta F_j = -F_j$, so the minimum is taken on a non-empty set,
and that the backward error is always well-defined. 
\end{remark}

In some frameworks it is important to assign weights to the perturbations. In our setting, this can be achieved considering a weighted matrix-valued function:
\[
\alpha_1 f_1(\lambda) F_1 + \ldots + \alpha_k f_k(\lambda) F_k,
\]
and assigning different values to the weights $\alpha_i$. For simplicity, we do not include this variant in our results, but a scale of this form is possible.

The backward error $\eta$ depends 
on the approximate eigenpairs. We do not explicitly report this 
dependence to ease the notation, and we assume that they have 
been fixed throughout this section.
\br
Note that in Definition \ref{def:unstructured-be} we consider approximate eigenpairs $(\hat \lambda_i, \hat v_i)$, assuming that $\hat v_i$ is an approximate eigenvector associated with the approximate eigenvalue $\hat \lambda_i$. 
In principle, we may give an alternative definition that does not 
enforce this matching, and simply ask that a family $\hat \lambda_i$ and $\hat v_i$ is made by 
eigenvalues and eigenvectors, possibly up to a permutation in the indices of 
the latter. However, 
we prefer to avoid this because it would make the theoretical derivations 
less clear, and in practice it is easy to reorder and preprocess the data
to find good approximate eigenpairs beforehand. Throughout this work, we 
assume that this has been done in advance.
\er

We now give an explicit characterization of $\eta$.
\begin{theorem}
\label{th:back_err_with_eigenvectors}
        Let $G, V$ be the following 
        matrices: 
        \[
        G := \begin{bmatrix}
            f_1(\hat\lambda_1) & \ldots & f_k(\hat\lambda_1) \\ 
            \vdots && \vdots \\ 
            f_1(\hat\lambda_p) & \ldots&  f_k(\hat\lambda_p) \\ 
        \end{bmatrix}, 
        \qquad 
        V := \begin{bmatrix}
            \\ 
            \hat v_1 & \dots & \hat v_p \\ 
            &&
        \end{bmatrix},
      \]
      and denote by $G \odot^T V^T$ the Khatri-Rao transpose product between $G$ and $V^T$. Then the backward error 
      $\eta$ is equal to 
       \[
        \eta = \left\| 
         R \left[ 
            (G \odot^T V^T)^\dagger \right]^T
        \right\|_F, \qquad 
        R:=\sum_{j=1}^k F_j V f_j(\Lambda),
      \]
      where we define the matrix 
      \[
      \Lambda:=\begin{bmatrix}
          \hat \lambda_1 & &\\
          & \ddots & \\
          & & \hat \lambda_p
      \end{bmatrix} \in \mathbb{C}^{p\times p}. 
      \]
      In particular
      $\eta \leq 
        \sigma_{\hat p}(G \odot^T V^T)^{-1} \|{R}\|_F$, 
        where $\hat p$ is the rank of $G \odot^T V^T$.
\end{theorem}

\begin{proof}
    The definition of $\eta$ involves perturbations 
    $\delta F_j$ such that the relation 
    $\sum_{j = 1}^k f_j(\hat \lambda_i) (F_j + \delta F_j) \hat v_i = 0$ holds for $i = 1, \ldots, p$. This is 
    a linear relation in $\delta F_j$, which can 
    be written in matrix form as follows:
    \begin{equation}
    \label{eq:starting_relation_unstr}
      \left( 
      \underbrace{\begin{bmatrix}
           f_1(\hat \lambda_1) \hat v_1^T & 
            \ldots &
         f_k(\hat \lambda_1) \hat v_1^T \\ 
          \vdots && \vdots \\ 
        f_1(\hat \lambda_p) \hat v_p^T & 
            \ldots &
          f_k(\hat \lambda_p) \hat v_p^T
      \end{bmatrix}}_{G \odot^T V^T} 
      \otimes I_n \right) 
      \begin{bmatrix}
          \mathrm{vec}(\delta F_1) \\ 
          \vdots \\ 
          \mathrm{vec}(\delta F_k)
      \end{bmatrix} = - 
      \begin{bmatrix}
            \sum_{j = 1}^k f_j(\hat \lambda_1) F_j \hat v_1 \\ 
          \vdots \\ 
            \sum_{j = 1}^k f_j(\hat \lambda_p) F_j \hat v_p 
      \end{bmatrix}.
    \end{equation}
    In view of Remark~\ref{rem:solvability} the above linear system
    admits at least a 
    non-trivial solution. The minimum Euclidean norm 
    solution is given by 
    \begin{equation}
    \label{eq:minim_norm_solution}
      \begin{bmatrix}
          \mathrm{vec}(\delta F_1) \\ 
          \vdots \\ 
         \mathrm{vec}(\delta F_k)
      \end{bmatrix} = - \left[ 
        (G \odot^T V^T)^\dagger \otimes I_n 
      \right] r, \quad \mbox{where} \; r:=
      \begin{bmatrix}
            \sum_{j = 1}^k f_j(\hat \lambda_1) F_j \hat v_1 \\ 
          \vdots \\ 
            \sum_{j = 1}^k f_j(\hat \lambda_p) F_j \hat v_p 
      \end{bmatrix}.
    \end{equation}
Using the properties of the Kronecker product, we may write
\[
\begin{bmatrix}
\vect{\delta F_1} \\
 \vdots \\
 \vect{\delta F_k}
\end{bmatrix} = -\vect{ R \left[ (G \odot^T V^T)^{\dagger}\right]^T}.
\]
Relation \eqref{eq:minim_norm_solution} gives the following upper bound on the backward error $\eta$:
    \[
      \eta := \left\| 
      \begin{bmatrix}
        \delta F_1 \\ 
        \vdots \\ 
        \delta F_k
      \end{bmatrix}
      \right\|_F \leq 
      \| 
        (G \odot^T V^T)^\dagger \otimes I_n
      \|_2
      \| r \|_2
      = \sigma_{\hat p}(
        G \odot^T V^T
      )^{-1} \|R \|_F,
    \]
    where $\hat p$ is the rank of $G \odot^T V^T$. \end{proof}

\br
\begin{remark}
Observe that the relation \eqref{eq:starting_relation_unstr} is equivalent to 
\[
\begin{bmatrix}
    \delta F_1 & \ldots & \delta F_k
\end{bmatrix} X = - \begin{bmatrix}
    F_1 & \ldots &  F_k
\end{bmatrix} X, \quad \mbox{with } X := \begin{bmatrix}
    f_1(\hat \lambda_1) \hat{v}_1 & \ldots &   f_1(\hat \lambda_p) \hat{v}_p \\
    \vdots & & \vdots \\
      f_k(\hat \lambda_1) \hat{v}_1  & \ldots &   f_k(\hat \lambda_p) \hat{v}_p 
\end{bmatrix}, 
\]
and this can be used to provide an alternative derivation of 
the formula in Theorem~\ref{th:back_err_with_eigenvectors} as a direct 
consequence of \cite[Lemma $1.3$]{Sun}. This lemma is 
proved in \cite{Sun} as a building block 
for the analysis of the related problem of characterizing the backward error
of generalized eigenvalue problems. 
\end{remark}
\er
 
The following result shows that the minimal norm perturbations $\delta F_j$ have a low-rank structure whenever the number of eigenpairs 
considered is small, that is $p \ll n$. 

\begin{lemma} \label{lem:low-rank-deltafj}
    The minimal norm backward errors $\delta F_j$ of Theorem~\ref{th:back_err_with_eigenvectors}
    can be expressed as $\delta F_j = - R M_j^T$, for appropriate $n \times p$
    matrices $M_j$, where $R$ is the $n \times p$ residual matrix
    \[
      R = \sum_{j = 1}^k F_j V f_j(\Lambda), \qquad \Lambda = \begin{bmatrix}
          \hat \lambda_1 \\
          & \ddots \\
          && \hat \lambda_p
      \end{bmatrix}.
    \]
\end{lemma}

\begin{proof}
    We denote by $M := (G \odot^T V^T)$, 
    and we partition its pseudoinverse $M^{\dagger}$ in $n \times 1$ blocks as follows:
    \[
      M^{\dagger} = \begin{bmatrix}
          m_{11} & \dots & m_{1p} \\ 
          \vdots && \vdots \\ 
          m_{k1} & \dots & m_{kp} \\
      \end{bmatrix}, \qquad 
      m_{ij} \in \mathbb C^n. 
    \]
    Then, we use the relation 
    $(m \otimes I_n) s = \mathrm{vec}(sm^T)$
    and by substituting 
    in the above relation we get 
    \[
      \delta F_j = 
        - \left( r_1 m_{j1}^T  + \ldots + r_p m_{jp}^T \right), \qquad 
        r_i := \sum_{j = 1}^k f_j(\hat \lambda_i) F_j \hat v_i. 
    \]
    Hence, all $\delta F_j$ are of rank at most $p$, and can be rewritten as 
    $\delta F_j = - R M_j^T$ for appropriate $n \times p$ matrices $M_j$.
\end{proof}

Note that the fact that the $\delta F_j$ are low-rank allows to easily compute their 
Frobenius and spectral norms (for instance by means of a reduced QR factorization of 
$R$ and $M_j$). 
In addition, the fact that they all share the same left factor $R$ implies that any linear 
combination of the $\delta F_j$ still has rank at most $p$. 

\subsection{Backward errors for the eigenvalues}

If the eigenvectors are not computed or not available, we may 
give another definition of backward error as follows:
\[
    \eta := \min_{\hat v_i \neq 0} 
      \min
       \Bigg\{ \| \left[ \delta F_1, \ldots, \delta F_k \right] \|_F \ \Big| \ 
       \exists \delta F_j, 
      \sum_{j = 1}^k f_j(\hat \lambda_i) (F_j + \delta F_j)
       \hat v_i =0
    \Bigg\}.
\]
This definition coincides with minimizing Definition~\ref{def:unstructured-be} over all possible choices of 
eigenvectors $\hat v_i$, since we are looking for the closest 
nonlinear eigenvalue problem with prescribed eigenvalues, and 
no constrained on the eigenvectors. Using a small abuse of notation, we denote by $\eta$ the backward error associated with the eigenvalues $\hat \lambda_i$, even when the eigenvectors are not computed.

We may provide a version of Theorem~\ref{th:back_err_with_eigenvectors} suited to this 
scenario. 
\begin{theorem}
\label{thm:bounds_for_back_error_without_eigenvectors}
   For $i=1,\ldots,p$, denote by $\hat u_i, \hat v_i$ respectively the left and right singular vectors of the matrix $\sum_{j=i}^k f_j(\hat \lambda_i) F_j$, associated with the smallest singular value, denoted by $\hat \sigma_i$.  Let $G$ be the matrix defined in Theorem \ref{th:back_err_with_eigenvectors} and $V$ be the following matrix:
    \[
    V := \begin{bmatrix}
            \\ 
            \hat v_1 & \dots & \hat v_p \\ 
            &&
        \end{bmatrix}. 
      \]
     If $G \odot^T V^T$ has rank $\hat p$, then we have the following upper 
     and lower bounds for $\eta$:
    \[
   \max_{i=1,\ldots,p} \left( \frac{\hat \sigma_i}{\sqrt{\sum_{j=1}^k \left| f_j(\hat \lambda_i) \right|^2}} \right) \leq \eta \leq \sigma_{\hat p}(
        G \odot^T V^T
      )^{-1} \sqrt{p} \max_{i=1,\ldots,p} \hat \sigma_i .
    \]
\end{theorem}
\begin{proof}
 We start by proving the upper bound for $\eta$. We consider the matrix relation
     \begin{align*}
    \begin{bmatrix}
         \mathrm{vec}(\delta F_1) \\ 
          \vdots \\ 
        \mathrm{vec}(\delta F_k)
      \end{bmatrix} &= - \left[ 
        (G \odot^T V^T)^\dagger \otimes I_n 
      \right]
      \begin{bmatrix}
            \sum_{j = 1}^k f_j(\hat \lambda_1) F_j \hat v_1 \\ 
          \vdots \\ 
            \sum_{j = 1}^k f_j(\hat \lambda_p) F_j \hat v_p 
      \end{bmatrix} \\
      &= - \left[ 
        (G \odot^T V^T)^\dagger \otimes I_n 
      \right]
      \begin{bmatrix}
         \hat \sigma_1 \hat u_1 \\ 
          \vdots \\ 
         \hat \sigma_p \hat u_p \\             
      \end{bmatrix},
     \end{align*}
from which we have the following upper bound:
\begin{align*}
    \eta \leq \sigma_{\hat p}(
        G \odot^T V^T
      )^{-1} \left\| \begin{bmatrix}
         \hat \sigma_1 \hat u_1 \\ 
          \vdots \\ 
         \hat \sigma_p \hat u_p \\             
      \end{bmatrix} \right\|_2 &= \sigma_{\hat p}(
        G \odot^T V^T
      )^{-1} \sqrt{ \sum_{i=1}^p \hat \sigma^2_i} \\
      & \leq \sigma_{\hat p}(
        G \odot^T V^T
      )^{-1} \sqrt{p} \max_{i=1,\ldots,p} \hat \sigma_i.
\end{align*}
For each $i=1,\ldots,p$, starting from the relation $\sum_{j=1}^k f_j(\hat \lambda_i) \delta F_j \hat v_i= - \hat \sigma_i \hat u_i$, we have that:
\begin{align*}
      \hat \sigma_i = \left\| \sum_{j=1}^k f_j(\hat \lambda_i) \delta F_j \hat v_i \right\|_F &\leq \sum_{j=1}^k \left| f_j(\hat \lambda_i) \right| \| \delta F_j \|_F \\
      &\leq \sqrt{\sum_{j=1}^k \left| f_j(\hat \lambda_i) \right|^2} \sqrt{\sum_{j=1}^k  \|\delta F_j \|^2_F} \\
      & \leq \sqrt{\sum_{j=1}^k \left| f_j(\hat \lambda_i) \right|^2} \eta.
\end{align*}
Then maximazing over $i=1,\ldots,p$, we obtain the following lower bound for $\eta$:
\[
\eta \geq  \max_{i=1,\ldots,p} \left( \frac{\hat \sigma_i}{\sqrt{\sum_{j=1}^k  \left| f_j(\hat \lambda_i) \right|^2}} \right).
\]
\end{proof}

For the case $p=1$, we obtain an explicit expression for the backward error $\eta$, which coincides with the one proposed by Ahmad and Mehrmann (Proposition $2.2$, \cite{AhmMehr}). 

\begin{corollary}
    For the case $p=1$, we have an explicit expression for the backward error
    \[
    \eta=\frac{\hat \sigma_1}{\sqrt{\sum_{j=1}^k  \left| f_j(\hat \lambda_i) \right|^2}},
    \]
    where $\hat \sigma_1$ denotes the smallest singular value of the matrix $\sum_{j=1}^k f_j(\hat \lambda_1) F_j$.
\end{corollary}

\begin{proof}
    Let $\hat v_1$ be the right singular vector associated with the singular value $\hat \sigma_1$. Then the upper bound for $\eta$ proposed in Theorem \ref{thm:bounds_for_back_error_without_eigenvectors} may be written as 
    \[
    \sigma_{\min}( G \odot^T \hat v_1^T)^{-1} \hat \sigma_1 = \|  G \otimes \hat v_1^T \|_2^{-1} \hat \sigma_1 = \| G \|_2^{-1} \hat \sigma_1 = \frac{\hat \sigma_1}{\sqrt{\sum_{j=1}^k \left| f_j (\hat \lambda_1) \right|^2} }.
    \]
\end{proof}

\subsection{Explicit upper bounds for the backward errors}
The backward errors depend on the norm of 
the pseudoinverse $G \odot^T V^T$, which is not 
necessarily easy or cheap to compute. In this section, 
we provide some upper bounds that can be used in place of 
computing the norm explicitly. 

\begin{lemma}
\label{lem:explicit_upper_bounds}
    Let $G$ be a $p \times k$ matrix and 
    $V$ be a $n \times p$ matrix, with $V$
    scaled to have $\| V e_i \|_2 = 1$ for $i = 1, \ldots, p$, where $e_i$ is the $i$-th vector of the canonical basis. 
    The following 
    bounds for the norm of $\| (G \odot^T V^T)^\dagger\|_2$ hold:
    \begin{itemize}
        \item If $p \leq kn$, then 
          $\| (G \odot^T V^T)^\dagger\|_2
          \leq \sigma_p(G)^{-1} \kappa_2(V)$, 
        \item If $p \leq k$, then $\| (G \odot^T V^T)^\dagger\|_2
          \leq \sigma_p(G)^{-1},$
    \end{itemize}
    where $\kappa(V) = \sigma_1(V) / \sigma_p(V)$ is the 
    condition number of $V$.
\end{lemma}
\begin{proof}
Let us denote by $M := G \odot^T V^T$. We first prove that if $p \leq kn$, then $\| M^\dagger\|_2 \leq \sigma_p(G)^{-1} \kappa_2(V)$. The condition 
    $p \leq kn$ implies that $\| M^\dagger\|_2 = \sigma_p(M)^{-1}$. 

    Let us denote by $J$ the $p \times p^2$ submatrix of $I_{p^2}$ 
    such that $J(G \otimes V^T) = G \odot^T V^T$; then, we have 
    \[
      \sigma_p(M) = \sigma_p(J (G \otimes V^T)) \geq 
        \sigma_{p^2}(G \otimes V^T) = 
        \sigma_p(G) \sigma_p(V), 
    \]
    where we have used the $J$ is a unitary projection and 
    therefore for any $W$ it holds $\sigma_{\min} (JW) \geq \sigma_{\min}(W)$, and the properties of the Kronecker product. Since 
    we assumed that $\| V e_i \|_2 = 1$
    which implies $\| V \|_2 \geq 1$, we have 
    $\sigma_p(V) \geq \kappa_2(V)^{-1}$, and we conclude. 
       
    To prove the second inequality, consider a 
    QR factorization of $G^T$, which has the form $Q\begin{bmatrix}
        R \\
        0
    \end{bmatrix} = G^T$ 
    for a $k \times k$ matrix $Q$ and an upper triangular $R$. 
    We then define $X$ as follows: 
    \[
      G^T = Q \begin{bmatrix}
          R \\ 0
      \end{bmatrix}, \quad 
      X = Q \begin{bmatrix}
          R^{-T} & 0 \\ 
          & \sigma_p(G)^{-1} I_{k-p}
      \end{bmatrix} \implies 
      GX = \begin{bmatrix}
          I_p & 0
      \end{bmatrix}. 
    \]  
We now right multiply $M$ by $X \otimes I_n$, which 
yields a matrix with the following block structure:
\[
  \hat M := M (X \otimes I_n) = 
    \begin{bmatrix}
        I_p & 0
    \end{bmatrix} \odot^T V^T = 
    \begin{bmatrix}
        v_1^T &&\\ 
        & \ddots & \\ 
        && v_p^T \\
    \end{bmatrix}, 
\]
where we have used the notation $v_i := Ve_i$. 
We have the following 
relation between singular values of $M$ and $\hat M$:
\[
\sigma_i(M) \geq 
  \sigma_i(\hat M) \sigma_p(X^{-1})
  = \sigma_i(\hat M) \sigma_1(X)^{-1}
  = \sigma_i(\hat M) \sigma_p(G). 
\]
We now prove that $\sigma_p(\hat M) \geq 1$, which concludes the 
proof. By the variational characterization of the singular values, 
we may write 
\begin{align*}
  \sigma_p(\hat M) &:= \min_{\|{w}\|_2 = 1} 
    \| w^T \hat M \|_2 = 
    \left\|
      \begin{bmatrix}
          w_1 v_1^T & \dots & w_p v_p^T & \times & \dots & \times 
      \end{bmatrix}
    \right\|_2 \\ 
    &\geq \left\|
      \begin{bmatrix}
          w_1 v_1^T & \dots & w_p v_p^T 
      \end{bmatrix}
    \right\|_2 
    = \sqrt{ \sum_{i = 1}^p |w_i|^2 \|{v_i}\|^2_2 } 
    = \sqrt{ \sum_{i = 1}^p |w_i|^2  } = 1,  
\end{align*}
where in the last steps we have used $\|{v_i}\|_2 = 1$ 
and $\|{w}\|_2 = 1$.

\end{proof}

\begin{remark}
    Note that in principle it may happen $\kappa(V)=\infty$ or $\sigma_p(G)=0$. In both cases, the statement of Lemma \ref{lem:explicit_upper_bounds} still holds and yields $\| (G \odot^T V)^{\dagger}\|_2 \leq \infty$.
\end{remark}

\section{Structured nonlinear eigenvalue problems}
\label{sec:structured}

In this section we propose an extension of our analysis that 
deals with the case when the coefficients $F_j$ have a specific structure that should be preserved in the backward error. For instance, they could be symmetric, Toeplitz, with a given sparsity pattern, or low-rank. Depending on the structure that we consider, we may need to provide a different approach for the 
computation of the backward error. More specifically, 
we make the assumption that $F_j \in \mathcal S_j \subseteq \mathbb C^{n \times n}$, where 
$\mathcal S_j$ is a set of matrices with a particular structure. 

We assume that $\mathcal{S}_j$ are at least differentiable manifolds 
that includes the zero. It is convenient to distinguish two cases:
\begin{enumerate}
    \item For linear structures, when all the $\mathcal S_j$ 
    are linear subspaces, we provide a formula for the structured backward error associated with a set of approximate eigenpairs; we will describe this case in Section~\ref{subsec:linear subspaces}. We will provide some results that hold 
    for symmetric matrices in Section~\ref{subsec:symmetric-backward}. 
    \item For nonlinear structures (such as fixed rank matrices), we propose an approximate upper bound, computed employing a Riemannian optimization-based approach; we will describe it in Section~\ref{sec:nonlinear-structures}. 
\end{enumerate}

\subsection{Structured coefficients in linear subspaces}
\label{subsec:linear subspaces}

If the sets $\mathcal S_j$ are linear subspaces of $\mathbb C^{n \times n}$, then we can 
write the $F_j$ in an appropriate basis:
\[
  F_j = \sum_{i = 1}^{d_j} \delta_i^j P^{(i,j)}, \qquad 
  \mathcal S_j = \mathrm{span}(
    P^{(1,j)}, \ldots, P^{(d_j,j)}
  ).
\]
From now on, we will assume that the basis given by the $P^{(i,j)}$ are
orthogonal with respect to the Frobenius inner product, and normalized to 
have Frobenius norm equal to $1$. This implies that the matrix 
\[
  P^{(j)} = \begin{bmatrix}
      \\
      \vect{P^{(1,j)}},\ \ldots,\ \vect{P^{(d_j,j)}} \\ 
      \\
  \end{bmatrix}
\]
has orthonormal columns. We denote by $P$ the block diagonal matrix collecting
all $P^{(j)}$, defined as follows:
\begin{equation} \label{eq:P}
  P := \begin{bmatrix}
      P^{(1)} \\ 
      & \ddots \\ 
      && P^{(k)}
  \end{bmatrix}.
\end{equation}
Note that since the $P^{(j)}$ are not square, the above matrix is rectangular as well, and has 
orthonormal columns. Throughout this section, particular results for the unstructured case can 
be obtained simply choosing $P = I$. 

\begin{definition}
Given $F_j \in \mathcal{S}_j$, for $j=1,\ldots,k$, where $\mathcal{S}_j$ are linear subspaces of $\mathbb{C}^{n \times n}$, consider $(V, \Lambda)$, defined as in Theorem \ref{th:back_err_with_eigenvectors}, approximate eigenpairs for the matrix-valued function $F(\lambda)=\sum_{j=1}^k f_j(\lambda) F_j$. The \emph{structured backward error} associated with $(V,\Lambda)$ is defined as:
\[
\eta_{\mathcal{S}}:= \min \left\lbrace \| \left[ \delta F_1,\ldots \delta F_k \right]\|_F : \sum_{j=1}^k \left( F_j +\delta F_j \right)  V f_j(\Lambda)=0, \; \delta F_j \in \mathcal{S}_j, \mbox{for} \; j=1,\ldots,k \right\rbrace .
\]
\end{definition}
We prove the structured analogue of Theorem~\ref{th:back_err_with_eigenvectors}.
\begin{theorem}
\label{th:struct_back_err_with_eigenpairs}
    Let $(V, \Lambda)$ approximate eigenpairs for the nonlinear eigenvalue problem with structured coefficients $F(\lambda)$, such that 
    \[
      R = \sum_{j = 1}^k F_j V f_j(\Lambda),
    \]
    and let $G$ be defined as in Theorem \ref{th:back_err_with_eigenvectors}.
    Then, the structured backward error $\eta_{\mathcal{S}}$ is equal to 
    \[
        \eta_{\mathcal{S}} = \| 
          \left[ 
          \left( \left( G \odot^T V^T \right) \otimes I_n \right) P \right]^\dagger r 
        \|_2, \qquad 
        r := \vect{R},
      \]            
      where $P$ is defined as in \eqref{eq:P}. 
      In particular, we have the upper bound 
      \[
      \eta_{\mathcal{S}} \leq 
        \sigma_{\min}(\left( \left( G \odot^T V^T \right) \otimes I_n \right) P)^{-1} \|{R}\|_F.
      \]
\end{theorem}

\begin{proof}
    Starting with relation \eqref{eq:starting_relation_unstr} provided in Theorem \ref{th:back_err_with_eigenvectors}, we get
    \[
    \left( \left( G \odot^T V^T \right) \otimes I_n \right) \begin{bmatrix}
            \vect{\delta F_1} \\
            \vdots \\
            \vect{\delta F_k} \\
    \end{bmatrix} = - r.
    \]
    We observe that $\vect{\delta F_j}= P^{\left( j\right)} \delta^j$, where $\delta^j=\begin{bmatrix}
    \delta_1^j & \ldots & \delta_{d_j}^j
\end{bmatrix}^T$, for $j=1,\ldots,k$ and therefore $\|\delta F_j \|_F=\|\delta^j \|_2$.  Then we may write the previous relation as
\[
   \left( \left( G \odot^T V^T \right) \otimes I_n \right) P \delta = - r, \qquad \delta=\begin{bmatrix}
       \delta^1 \\
       \vdots \\
       \delta^k
   \end{bmatrix}.
\]
Since $\|\delta\|_2=\| \begin{bmatrix}
    \delta F_1 & \ldots & \delta F_k
\end{bmatrix}\|_F$, we conclude the proof. The upper bound for $\eta_{\mathcal{S}}$ follows from:
\[
\| \delta \|_2= \left\| 
      \begin{bmatrix}
        \delta F_1 \\ 
        \vdots \\ 
        \delta F_k
      \end{bmatrix}
      \right\|_F \leq 
      \| 
         \left[ \left( \left( G \odot^T V^T \right) \otimes I_n \right) P \right]^{\dagger}
      \|_2
      \| r \|_2
      = \sigma_{\min}(
         \left( \left( G \odot^T V^T \right) \otimes I_n \right) P )^{-1} \|R \|_F. \qedhere
\]
\end{proof}

\subsubsection{Invariant pairs}
    It is possible to provide a generalization of Theorem \ref{th:struct_back_err_with_eigenpairs} using the notion of invariant pairs. Given a nonlinear matrix-valued function $F(\lambda) = \sum_{j = 1}^k F_j f_j(\lambda)$, 
we say that $(V, M)$ is an invariant pair if the following relation holds:
\[
  \sum_{j = 1}^k F_j V f_j(M) = 0. 
\]
Note that this implies that $\Lambda(M)$ is a subset of the spectrum of $F(\lambda)$ and 
that the associated eigenvectors belong to the column span of $V$.
Besides being useful for analyzing (for instance) stable subspaces, 
this also allows to maintain real arithmetic in case of 
complex conjugate eigenvalues. 

In this setting, denoting by
\[
\widehat G:= \begin{bmatrix}
        f_1(M)^T & \cdots & f_k(M)^T
\end{bmatrix}, \qquad R:= \sum_{j=1}^k F_j V f_j(M),
\]
and proceeding as in the proof of Theorem \ref{th:struct_back_err_with_eigenpairs}, we have that:
\[
\eta_{\mathcal{S}} = \left\| 
          \left[ \left( \left(
            \widehat G (I_k \otimes V^T) 
          \right) \otimes I_n \right) P \right]^\dagger r 
        \right\|_2, \qquad 
        r := \vect{R},
\]
and consequently the upper bound
\[
\eta_{\mathcal{S}} \leq 
        \sigma_{\min}(\left( \left(
            \widehat G (I_k \otimes V^T) 
          \right) \otimes I_n \right) P)^{-1} \|{R}\|_F.
\]

Even though Theorem~\ref{th:struct_back_err_with_eigenpairs}
provides an explicit formula for the backward error, the linear system that needs to be 
solved is much larger than the one in the non-structured case. Hence, it is sometimes convenient
to obtain the backward error through the same optimization procedures that we will describe for 
nonlinear structures in Section~\ref{sec:nonlinear-structures}.

\subsection{Symmetric backward errors}
\label{subsec:symmetric-backward}

For particular choices of $\mathcal S_j$, we can provide a more detailed analysis. We now 
focus on the case of real symmetric coefficients $F_j = F_j^T$. For the standard eigenvalue problem, in \cite{TisseurChart} Tisseur provides a complete survey on structured backward errors associated with multiple approximate eigenpairs. In particular, a formula
is provided for computing
the symmetric backward error, using the result on structured matrix problems in \cite[Lemma~2.3]{TisseurChart}. Our result 
is a generalization to the context of nonlinear eigenvalue problems with symmetric coefficients.

For simplicity, we only discuss the case of eigenvalues and eigenvectors, even though the 
same analysis can be generalized to invariant pairs with a moderate effort. 
The problem can be stated as, given a $n \times p$ matrix $V$ and a diagonal matrix 
$\Lambda$, finding real symmetric perturbations $\delta F_j = \delta F_j^T$ such that 
\begin{equation} \label{eq:fjeq}
  \sum_{j = 1}^k (F_j + \delta F_j) V f_j(\Lambda) = 0.
\end{equation}
\begin{theorem} \label{thm:backward-error-symmetric}
    Let $F_j$ be real symmetric matrices, and $(V, \Lambda)$  approximate eigenpairs with  
    a diagonal matrix $\Lambda$ such that $R = \sum_{j = 1}^k F_j V f_j(\Lambda)$. Let $V = QT$
    be an economy size QR factorization of $V$, and define 
    \[
      \tilde T = \begin{bmatrix}
          T f_1(\Lambda) \\ \vdots \\ 
          T f_k(\Lambda) \\
      \end{bmatrix}, \qquad 
      M_S := \left[ \begin{array}{ccc}
          &  \tilde T^T \otimes I_p & \\ \hline 
          \Pi_{p,p} - I_{p^2} && \\
          & \ddots & \\
          && \Pi_{p,p} - I_{p^2}
      \end{array}  \right], 
    \]    
    where $\Pi_{p,p}$ is the $(p,p)$ commutation matrix (or perfect shuffle, see \cite{van2000ubiquitous}). 
    Then, there 
    exist symmetric real perturbations $\delta F_j$ such that $(V, \Lambda)$ are eigenpairs 
    for $\sum_{j = 1}^k (F_j + \delta F_j) f_j(\lambda)$, and 
\[
\delta F_j = Q \begin{bmatrix}
    A_{11}^{\left(j\right)} &    \left( A_{21}^{\left(j\right)} \right)^T \\
    A_{21}^{\left(j\right)}  & 0_{\left( n-p \right)  \times \left( n-p \right)}\\
\end{bmatrix} Q^T,
\]
where $A_{11}^{\left( j\right)}$ and $A_{21}^{\left( j\right)}$ solve the equations:
\[
      \begin{bmatrix}
          A^{(1)}_{21} & \dots & A^{(k)}_{21}
      \end{bmatrix} = B_2 \tilde T^\dagger, \quad \mbox{and} \quad
     M_S \left[ 
        \begin{array}{c}
             \vect{A_{11}^{(1)}} \\
             \vdots \\
             \vect{A_{11}^{(k)}}
        \end{array}
      \right] = 
      \left[ 
        \begin{array}{c}
             \vect{B_1} \\ \hline 
             0 \\
             \vdots \\
             0 \\
        \end{array}
      \right],
\]
where we denote by $B_1,B_2$ the $p\times p$ and $\left( n-p\right) \times p$ blocks of $\begin{bmatrix}
    B_1 \\
    B_2
\end{bmatrix}=-Q^T R$, respectively.
\end{theorem}

\begin{proof}
We note that choosing $\delta F_j = -F_j$ gives a valid solution to the linear system; hence, we know 
a-priori that the set of all solutions is non-empty, and we look for the minimum norm one. 
    As a preliminary step, we choose a unitary matrix $Q$ such that 
    \[
      V f_j(\Lambda) = Q T_j = Q \begin{bmatrix}
          \tilde T_j \\ 
          0_{(n-p) \times p}
      \end{bmatrix}, \qquad 
      \tilde T_j \in \mathbb C^{p \times p},
    \]
    with upper triangular matrices $\tilde T_j$. Since $f_j(\Lambda)$ is diagonal, such $Q$ can be 
    constructed from a QR factorization of $V$. We now left multiply \eqref{eq:fjeq} by $Q^T$
    obtaining 
    \[
      \sum_{j = 1}^k Q^T \delta F_j Q T_j = - \sum_{j = 1}^k Q^T F_j Q T_j =: - Q^T R, 
    \]
    where we have used $Q^T V f_j(\Lambda) = T_j$. By partitioning $Q^T \delta F_j Q$
    as follows 
    \[
      Q^T \delta F_j Q = \begin{bmatrix}
          A^{(j)}_{11} & (A^{(j)}_{21})^T \\
          A^{(j)}_{21} & A^{(j)}_{22} \\
      \end{bmatrix}, \qquad 
      - Q^T R = \begin{bmatrix}
          B_1 \\ B_2
      \end{bmatrix} ,
    \]
    we can rewrite the previous equation as 
    \[
      \sum_{j = 1}^k Q^T \delta F_j Q T_j = 
      \sum_{j = 1}^k \begin{bmatrix}
          A^{(j)}_{11} & (A^{(j)}_{21})^T \\
          A^{(j)}_{21} & A^{(j)}_{22} \\
      \end{bmatrix} \begin{bmatrix}
          \tilde T_j \\ 
          0_{(n-p) \times p}
      \end{bmatrix} = 
      \sum_{j = 1}^k \begin{bmatrix}
          A^{(j)}_{11}  \\
          A^{(j)}_{21}  \\
      \end{bmatrix} 
      \tilde T_j = \begin{bmatrix}
          B_1 \\ B_2
      \end{bmatrix} .
    \]
    The only condition to have a symmetric solution is to ensure that 
    $A^{(j)}_{11} = (A^{(j)}_{11})^T$ for all $j = 1, \ldots, k$, 
    and the above equation 
    decouples in the independent linear systems 
    \begin{equation} \label{eq:Atj}
      \sum_{j = 1}^k 
          A^{(j)}_{11} 
      \tilde T_j = 
          B_1, \qquad 
    \sum_{j = 1}^k
          A^{(j)}_{21}
      \tilde T_j = B_2.
    \end{equation}
    Note that \[
      \| Q^T \delta F_j Q \|_F^2 = 
        \| A^{(j)}_{11}  \|_F^2 + 
        2 \| A^{(j)}_{21}  \|_F^2 + 
        \| A^{(j)}_{22}  \|_F^2. 
    \]
    Since we are looking for the minimum norm solution, we can choose 
    $A^{(j)}_{22}=0$, and $A^{(j)}_{21}$ as the minimum norm solution of 
    the right equation in \eqref{eq:Atj}:
    \begin{equation}
    \label{eq:equation_block21}
      \begin{bmatrix}
          A^{(1)}_{21} & \dots & A^{(k)}_{21}
      \end{bmatrix} = B_2 \tilde T^\dagger, \qquad 
      \tilde T := \begin{bmatrix}
          \tilde T_1 \\ 
          \vdots \\
          \tilde T_k
      \end{bmatrix}.
    \end{equation}

\br
As we argued at the beginning of the proof, the linear system \eqref{eq:fjeq} is consistent, with a solution given by $\delta F_j = -F_j$. Since the decoupled equations \eqref{eq:Atj} are equivalent to \eqref{eq:fjeq}, they also admit a solution. Multiplying by the pseudoinverse 
of $T$ yields the minimum norm solution among all the 
minimizers (in Frobenius norm), so in this case we also 
have that $B_2 \tilde T^\dagger \tilde T = B_2$. 
\er
    
    To determine $A^{(j)}_{11}$, we write a linear system with the left equation 
    in \eqref{eq:Atj} together with the symmetry condition 
    $A^{(j)}_{11} = (A^{(j)}_{11})^T$. This yields 
    \begin{equation}
    \label{eq:equation_block11}
    \underbrace{
      \left[ \begin{array}{ccc}
          &  \tilde T^T \otimes I_p & \\ \hline 
          \Pi_{p,p} - I_{p^2} && \\
          & \ddots & \\
          && \Pi_{p,p} - I_{p^2}
      \end{array}  \right]}_{ =: M_S}
      \left[ 
        \begin{array}{c}
             \vect{A_{11}^{(1)}} \\
             \vdots \\
             \vect{A_{11}^{(k)}}
        \end{array}
      \right] = 
      \left[ 
        \begin{array}{c}
             \vect{B_1} \\ \hline 
             0 \\
             \vdots \\
             0 \\
        \end{array}
      \right],
   \end{equation}
    where $\Pi_{p,p}$ is the commutation matrix (or perfect shuffle) 
    such that $\Pi_{p,p} \vect{X} = \vect{X^T}$ \cite{van2000ubiquitous}. 
    We know that the system is solvable, so we can characterize the minimum norm solution by taking 
    the pseudoinverse of the matrix on the left.
\end{proof}

\begin{corollary}
\label{cor:corollary_bound_symm}
Under the hypotheses and the notation of Theorem~\ref{thm:backward-error-symmetric}, we have the following upper bound for the structured backward error associated with the approximate eigenpairs 
$(V,\Lambda)$:
    \[
      \eta \leq 
      \| R \|_F^2 \left( \| M_S^\dagger \|_F^2 + 2 \| \tilde T \|_F^2 \right). 
   \]
\end{corollary}

\begin{proof}
    Consider the minimum solution $\delta F_j$, given by Theorem~\ref{thm:backward-error-symmetric}. The relation \eqref{eq:equation_block11} yields to the upper bound  $     \| A_{11}^{(1)} \|_F^2 + \ldots + \| A_{11}^{(k)} \|_F^2 \leq 
      \| B_1 \|_F^2 \cdot \| 
        M_S^\dagger
      \|_F^2. 
    $
    Combining this and the expression on $A_{21}^{(j)}$ in \eqref{eq:equation_block21}, we obtain 
    \begin{align*}
      \sum_{j = 1}^k \| \delta F_j \|_F^2 &= 
        \sum_{j = 1}^k \| Q^T \delta F_j Q \|_F^2 
        = \sum_{j = 1}^k \left( \| A_{11}^{(j)} \|_F^2 + 2 \| A_{21}^{(j)} \|_F^2 \right) \\
        &\leq 
        \| R \|_F^2 \left(
          \| M_S^\dagger \|_F^2 + 2 \| \tilde T \|_F^2
        \right). \qedhere
    \end{align*}
\end{proof}

\begin{remark}
    In Corollary~\ref{cor:corollary_bound_symm}, a bound for the backward error can be computed 
    cheaply, whenever $p$ is small. Indeed, the matrices $M_S$ and $\tilde T$ can be computed with 
    $\mathcal O(np^2 + p^6 k^3)$. A lower complexity in $p$ may be achieved exploiting the structure 
    of $M_S$. The dominant term is $np^2$ as long as $p^4 k^3 < n$, which is realistic in large scale  
    applications where only a few eigenmodes are necessary. 
\end{remark}

\subsection{Nonlinear structures} \label{sec:nonlinear-structures}

In this section we describe the more general case of coefficients
$F_j$ belonging to a differentiable manifold, which may not be a linear subspace. One of the most relevant examples is taking 
$F_j \in \mathbb{R}_{r_j}^{n \times n}$, where we denote by $\mathbb{R}_{r_j}^{n \times n}$ the set of real matrices of rank $r_j$ and size $n \times n$. 

In this context, we may not write the perturbations $\delta F_j$ as linear combinations of a set of matrices $P^{\left( 1\right)}, \ldots, P^{\left( d_j \right)}$. 
In this case, we have that each matrix coefficient $F_j$ and each perturbed coefficient $F_j + \delta F_j$ belong to a manifold $\mathcal{S}_j \subseteq \mathbb{R}^{n \times n}$. Given $F_j \in \mathcal{S}_j$, we can rephrase the definition of backward error in the following way:
\[
\eta_{\mathcal{S}} =  \min \left\lbrace \| \left[ \delta F_1,\ldots, \delta F_k \right] \|_F : \;  \sum_{j=1}^k \left( F_j + \delta F_j \right) V f_j(\Lambda)=0, \; (F_j + \delta F_j) \in \mathcal{S}_j \right\rbrace,
\]
where $(V,\Lambda)$ contains approximate eigenpairs, as defined as in Subsection \ref{subsec:linear subspaces}.

\begin{remark}
Note that the previous definition coincides with the one provided in Subsection \ref{subsec:linear subspaces} for linear structures, where we have that $\delta F_j \in \mathcal{S}_j$.     
\end{remark}

In this setting, we may not provide an explicit formula for the structured backward error, as in Subsection \ref{subsec:linear subspaces}. Nevertheless, we may numerically approximate an upper bound for the structured backward error, employing Riemannian optimization. Denote by $\tilde F_j:= F_j + \delta F_j$, for $j=1,\ldots,k$. Consider a $\mu >0$, we may define the functional
\begin{align}
\label{eq:functional_to_minimize}
f : \mathcal{S}_1 \times \ldots \times \mathcal{S}_k &\mapsto \mathbb{R}\\
    f\left( \tilde F_1,\ldots,\tilde F_k \right) & \mapsto \|\sum_{j=1}^k \tilde F_j V f_j(\Lambda) \|_F^2 + \mu \|\begin{bmatrix}
    \tilde F_1 - F_1 & \ldots & \tilde F_k - F_k
\end{bmatrix} \|^2_F. \notag
\end{align}
The parameter $\mu$ is needed to force the optimization 
algorithm to find a minimum norm solution. Therefore, an upper bound for the structured backward error may be obtained minimizing the functional $f$ on the product manifold $\mathcal{S}:=\mathcal{S}_1 \times \ldots \times \mathcal{S}_k$. This setting allows us to employ Riemannian optimization, minimizing the function $f$ on the product manifold $\mathcal{S}$. We chose to implement this idea relying on \texttt{manopt}, a MATLAB package for Riemannian optimization \cite{BoumalMishraAbsil}, and in particular its implementation of the Riemannian trust region method. To this end, we recall a few results on the product manifolds and perform the computation of the Riemannian gradient and the Riemannian Hessian on $\mathcal{S}$, which we need for the trust-region method on Riemannian manifolds \cite{AbsilTrustregion}.

The product manifold $\mathcal{S}$ can be treated working separately on the manifolds $\mathcal{S}_j$. Indeed, the tangent space of $\mathcal{S}$ can be defined as
\[
T_{\left( \tilde {F}_1, \ldots, \tilde {F}_k \right)} \left( \mathcal{S} \right) := T_{\tilde F_1}(\mathcal{S}_1) \times \ldots \times T_{\tilde F_k}(\mathcal{S}_k),
\]
and the scalar product that we consider on it is the one inherited from the products on $\mathcal{S}_j$ for $j=1,\ldots,k$, that is:
\[
\left\langle (u_1, \ldots, u_k), (w_1,\ldots,w_k) \right\rangle^{\mathcal{S}}_{\left(\tilde{F}_1,\ldots, \tilde{F}_k \right)} := \left\langle u_1,w_1
\right\rangle^{\mathcal{S}_1}_{\tilde F_1}+ \ldots  +  \left\langle u_k,w_k
\right\rangle^{\mathcal{S}_k}_{\tilde F_k},
\]
where $(u_1, \ldots, u_k), (w_1,\ldots,w_k) \in T_{\left( \tilde F_1,\ldots, \tilde F_k \right)} \left( \mathcal{S}\right)$ and $\left\langle \cdot, \cdot \right\rangle^{\mathcal{S}_j}$ is the product associated with $\mathcal{S}_j$. In this setting, we consider only embedded manifolds $\mathcal{S}_j \subseteq \mathbb{R}^{n \times n}$ for $j=1,\ldots,k$, then the product is the real inner product $\left\langle u,v \right\rangle:=\mbox{trace}\left( u^T v\right)$.

Both the Riemannian gradient and the Riemannian Hessian for the function $f$ can be computed starting from the Euclidean ones. In particular, the Riemannian gradient is obtained computing the orthogonal projection of the Euclidean gradient of $f$ (here we denote by $f$ the smooth extension of the functional \eqref{eq:functional_to_minimize} to the ambient space) onto the tangent space $T(\mathcal{S})$. The computation of the Riemannian Hessian of $f$ needs both the Euclidean gradient and the Euclidean Hessian for $f$ and it can be obtained through the Weingarten map (see \cite[Section 5]{Boumal} for the details).

Even though the projection of both the gradient and the Hessian on the product manifold $\mathcal{S}$ is handled automatically in \texttt{manopt}, we will need to implement this carefully to make it efficient. To this end, we first need to derive the Euclidean gradient and the Euclidean Hessian of the functional $f$.

It is convenient to write the functional as
\begin{align*}
    f( \tilde F_1, \ldots, \tilde F_k )= \left\langle \tilde FW, \tilde F W \right\rangle + \mu \left\langle \tilde F - F, \tilde F - F\right\rangle,
\end{align*}
where
\[
\tilde F:=\begin{bmatrix}
    \tilde F_1 & \ldots & \tilde F_k
\end{bmatrix}, \; F:=\begin{bmatrix}
    F_1 & \ldots & F_k
\end{bmatrix} \; \mbox{and} \; W:=\begin{bmatrix}
    V f_1(\Lambda) \\
    \vdots \\
    V f_k(\Lambda) \\
\end{bmatrix}.
\]
In order to compute the Euclidean gradient of the functional $f$, we perform the directional derivative of $f$:
\[
\mbox{D} f(\tilde F)[\tilde E]=\frac{d}{dt} 
    f( \tilde F_1 + t \tilde E_1, \ldots, \tilde F_k + t \tilde E_k) \bigg|_{t=0},
\]
in the direction $\tilde E:= \begin{bmatrix}
    \tilde E_1 & \cdots & \tilde E_k
\end{bmatrix}$. In this case, we may write
\begin{align*}
    \frac{d}{dt} 
    f( \tilde F_1 + t \tilde E_1, \ldots, \tilde F_k + t \tilde E_k)\bigg|_{t=0} &= \sum_{j=1}^k 2 \left\langle \tilde E_j W_j, \tilde F W \right\rangle + 2\mu \left\langle \tilde E_j, \tilde F_j - F_j\right\rangle \\
    &= \sum_{j=1}^k \left\langle \tilde E_j , 2\tilde F W W_j^T + 2\mu (\tilde F_j -F_j)\right\rangle,
\end{align*}
\br where \er $W_j$ is the $j$-th block row of $W$ and we used the circulant property of the trace in the last step. Then using that the Euclidean gradient is the unique vector such that
\[
\forall \tilde F, \tilde E \in \mathbb{R}^{n\times n} \times \ldots \times \mathbb{R}^{n\times n} \qquad \mbox{D} f(\tilde F)[\tilde E] = \left\langle \tilde E, \mbox{grad} f(\tilde F) \right\rangle,
\]
we conclude that $\mbox{grad} f(\tilde F)= 2\tilde F W W^T + 2\mu (\tilde F - F)$.
In the implementation of the Euclidean gradient in \texttt{manopt}, it is useful to split the contributions for each term of the product manifold. Then we may consider  $\mbox{grad}_{\tilde F_j} f(\tilde F)= 2\tilde FW W_j^T + 2\mu (\tilde F_j - F_j)$. 

The Euclidean Hessian of the function $f:\mathbb{R}^{n\times n} \times \ldots \times \mathbb{R}^{n \times n} \mapsto \mathbb{R}$ at the point $\left( \tilde F_1, \ldots, \tilde F_k \right)$ is defined as the directional derivative of the Euclidean gradient $\mbox{grad} f$
\[
\mbox{Hess} f ( \tilde F) [ \tilde E ] = \mbox{D}\, \mbox{grad} f ( \tilde F)[\tilde E] = \lim_{t \rightarrow 0} \frac{\mbox{grad}f(\tilde F + t \tilde E) - \mbox{grad}f(\tilde F)}{t},
\]
where $\tilde E \in \mathbb{R}^{n\times n} \times \ldots \times \mathbb{R}^{n\times n}$. Note that here we denote by $f$ both the functional \eqref{eq:functional_to_minimize} and its smooth extension to the ambient space. Inside the Riemannian Trust Region 
scheme, we only need to evaluate the Euclidean Hessian along a specified direction $\tilde E = ( \tilde E_1, \ldots, \tilde E_k)$, which is given by the directional derivative 
\[
\frac{d}{dt} \mbox{grad} f(\tilde F_1 + t \tilde E_1,\ldots, \tilde F_k + t \tilde E_k) \bigg|_{t=0} = 2 \tilde E WW^T + 2\mu \tilde E.
\]
Moreover, observing that 
\[
\mbox{D} \, \mbox{grad} f (\tilde F) [\tilde E] := \left( \mbox{D}\, \mbox{grad}_{\tilde F_1} f (\tilde F) [\tilde E], \ldots, \mbox{D}\, \mbox{grad}_{\tilde F_k} f (\tilde F) [\tilde E] \right),
\]
we may split the contributions of the different terms of the product manifold for the implementation in \texttt{manopt}, obtaining that 
\[
\mbox{D}\, \mbox{grad}_{\tilde F_j} f (\tilde F) [\tilde E] = 2 \tilde E W W_j^T + 2\mu \tilde E_j.
\]

Note that both the (Euclidean) gradient and the (Euclidean) Hessian have a first term which is low-rank. Indeed both $2\tilde F W W^T$ and $2 \tilde E WW^T$ are expressed in a low-rank format, therefore for several choices of manifolds we may compute their projection directly in an efficient way. 

In Subsection \ref{sub:experiments_riemannian}, we test this approach for a selected number of structures. In particular, we consider the case of sparse matrices, multiples of the identity and fixed rank matrices. For these structures, once we have computed the matrix $\tilde F W$, we handle the projection of the term $\tilde F W W_j^T$ as follows:
\begin{enumerate}
    \item $\mathcal{S}_j$ is the set of sparse matrices in $\mathbb{R}^{n \times n}$: let $\mathcal{J} \subseteq \left\lbrace 1,\ldots, n \right\rbrace^2$ the set of indices corresponding to the nonzero entries of $\mathcal{S}_j$, then the matrix-matrix multiplication between $\tilde FW$ and $W_j^T$ as
    \[
     (\tilde FW W_j^T)_{(a,b)} = \left\lbrace \begin{array}{c c}
        \sum_{c=1}^p (\tilde FW)_{(a,c)} (W_j)_{(b,c)} & (a,b) \in \mathcal{J}  \\
          0 & (a,b) \not\in \mathcal{J}
     \end{array}\right. .
    \]
    The complexity for this product is $\mathcal{O}(p \left|  \mathcal{J}\right|)$.
    \item $\mathcal{S}_j$ is the set of matrices multiple of the identity in $\mathbb{R}^{n \times n}$: we can perform the projection of the matrix $\tilde F W W_j^T$ onto the tangent space of this manifold simply computing the $\frac{1}{n}\mbox{trace}(\tilde F W W_j^T)$, for which the computational cost is $\mathcal{O}(n p^2 + n)$;
    \item $\mathcal{S}_j$ is the set of matrices of fixed rank $r_j$ in $\mathbb{R}^{n \times n}$: a rank $p$ matrix is represented as $USV^T$ by storing 
    a structure with three fields $U, V \in \mathbb{R}^{n\times p}$, $S \in \mathbb{R}^{p \times p}$, where $U,V$ are orthonormal and the matrix $S$ is any diagonal or full-rank matrix. The term $\tilde F W W_j^T$ can be represented in this way by the matrices $\tilde F W$, $W_j$ and $I_p$, respectively. 
    The latter can be projected on the tangent space of $\mathbb{R}_{r_j}^{n\times n}$ by \texttt{manopt}
    using an economy-size SVD, which requires 
    $\mathcal{O}(nr_j(p + r_j))$ flops.
\end{enumerate}

The same procedures can be repeated for handling the projection of the matrix $\tilde E W W_j^T$, for these three manifolds.

In the numerical implementation of the method, we successively solve minimization problems in the form \eqref{eq:functional_to_minimize}, for different choices of the parameter $\mu$. This approach, also known in optimization theory as penalization method, consists in solving the problem for smaller and smaller choices of the parameter $\mu$, using the solution of one step as initial point for the following one. An overview on these call of solvers for constrained optimization is contained in \cite[Section 4]{Bertsekas}, while results on their generalization to Riemannian manifolds can be found in \cite{LiuBoumal}.

We describe the approach in Algorithm \ref{Alg:riemannian}, where each minimization problem needs to be solved using the Riemannian based-method proposed in this Section.

\algblock{Begin}{End}
\begin{algorithm}[!h]
\caption{Riemannian optimization-based algorithm}
\label{Alg:riemannian}
\hspace*{\algorithmicindent} \textbf{Input}: Matrices $F:=\left[ F_1 \, \cdots \, F_k \right]$, manifold $\mathcal{S}$, functions $f_i$, $(V,\Lambda)$ approximate eigenpairs, desired accuracy $\epsilon$
\\
\hspace*{\algorithmicindent} \textbf{Output}: $\eta_{\mathcal{S}}$ upper bound for the structured backward error, decrease factor $\rho$
\begin{algorithmic}[1]
\Begin
\State Set $\mu=1$
\State Set starting point $\tilde F = F$
\While{$\sqrt{\mu} > \epsilon$}
\State $\tilde F \gets \arg \min_{F \in \mathcal{S}} f$, in \eqref{eq:functional_to_minimize}
\State $\mu \gets  \rho \mu $
\EndWhile
\State $\eta_{\mathcal{S}} \gets \| \tilde F - F \|_F$
\End
\end{algorithmic}
\end{algorithm}

\begin{remark}
In Line $6$ of Algorithm \ref{Alg:riemannian}, we suggest decreasing the parameter $\mu$ by a constant factor $\rho$ at each step. Our choice for the factor $\rho$ and additional implementation details are available at the Github repository \url{https://github.com/miryamgnazzo/backward-error-nonlinear}.
\end{remark}

It is common to measure the backward errors using the Euclidean distance, which means computing the norm $\left\|\begin{bmatrix}
    \widetilde F_1 - F_1 & \ldots & \widetilde F_k - F_k \end{bmatrix} \right\|_F$, as we propose in our analysis. However, in some situations it is possible to measure the distance between points on the manifold, using the Riemannian distance. Following the idea in \cite[Theorem 3.1]{Jacobsson}, we could rephrase our upper bounds, working directly on the distance between points on the manifold. Since \br this result employs \er lower bounds on sectional curvatures, it strictly depends on the geometry of the manifold we use. Indeed, while for flat manifolds the sectional curvature is zero, for different manifolds it can be derived from principal curvatures. In particular, for the case of manifolds of matrices of fixed rank, it can be obtained using \cite[Theorem 24]{Feppon}.

\section{Numerical experiments}
\label{sec:numerical experiments}

This section is devoted to assessing the quality of the 
theoretical bounds, and to check the effectiveness of the 
Riemannian optimization scheme in computing the backward
errors. We also include tests for symmetric nonlinear 
eigenvalue problems as described in Section \ref{subsec:symmetric-backward}. For the case of nonlinear structures, our implementation of pseudocode \ref{Alg:riemannian} in MATLAB is freely available at \url{https://github.com/miryamgnazzo/backward-error-nonlinear}, together with the codes for the bounds in Section \ref{sec:unstructured} and Subsections \ref{subsec:linear subspaces}, \ref{subsec:symmetric-backward}.

Throughout this section, all nonlinear problems for which we need 
a few eigenvalues to test have been solved 
with the Newton method initialized with different starting points. The experiments were run using MATLAB 2022b on Intel Core i7-1070H.

\subsection{Unstructured tests}

\subsubsection{The Hadeler problem}

We consider the nonlinear eigenvalue problem in the form:
\begin{equation}
\label{eq:hadeler}
F(\lambda)v=\left[\left( e^{\lambda} -1 \right)A_2 + \lambda^2 A_1 - \alpha A_0\right]v=0,
\end{equation}
where the coefficient matrices $A_i \in \mathbb{R}^{8 \times 8}$ are symmetric and $\alpha=100$. This example is known as the Hadeler problem \cite{Hadeler} and it is part of the collection of nonlinear eigenvalue problems in the MATLAB package \texttt{nlvep} \cite{NLEVP}. We consider a set of $p=3$ approximate eigenpairs of \eqref{eq:hadeler} and randomly generate a set of $1000$ perturbation matrices $\delta A_j$ for $j=0,1,2$. Then we may compute the backward errors using the formula in Theorem \ref{th:back_err_with_eigenvectors} and test the upper bounds for the unstructured backward error provided in Theorem \ref{th:back_err_with_eigenvectors} and Lemma \ref{lem:explicit_upper_bounds}. 

\begin{figure}[h]
     \centering
     \begin{tikzpicture}
       \begin{loglogaxis}[
             legend pos = north west,
             width=.9\linewidth,
             ymax=10^-1,
             height=75mm,
             title={Backward error for the Hadeler problem \eqref{eq:hadeler}},
             xlabel={$\| R \|_F$}
                      ]
         \addplot[only marks, red, mark size = 1pt] table[y index = 1] {unstructured_bounds_check_hadeler_n.dat};
         \addplot[no marks, black, dashed] table[y index = 2] {unstructured_bounds_check_hadeler_n.dat};
        \addplot[no marks, blue, dashed] table[y index = 3] {unstructured_bounds_check_hadeler_n.dat};
         \addplot[no marks, green, dashed] table[y index = 4] {unstructured_bounds_check_hadeler_n.dat};
         \legend{Backward error, $\sigma_p(G)^{-1}\|R\|_F$, $\sigma_p(G)^{-1} \kappa_2(V)\| R\|_F$, $\sigma_{\hat p}(G \odot^T V^T)^{-1}\|R \|_F$}
       \end{loglogaxis}
     \end{tikzpicture}
      \caption{Comparison among the upper bounds for the unstructured backward error in Theorem \ref{th:back_err_with_eigenvectors} and Lemma \ref{lem:explicit_upper_bounds}, applied to the Hadeler problem \eqref{eq:hadeler}.}
     \label{fig:hadeler_figure}
 \end{figure}
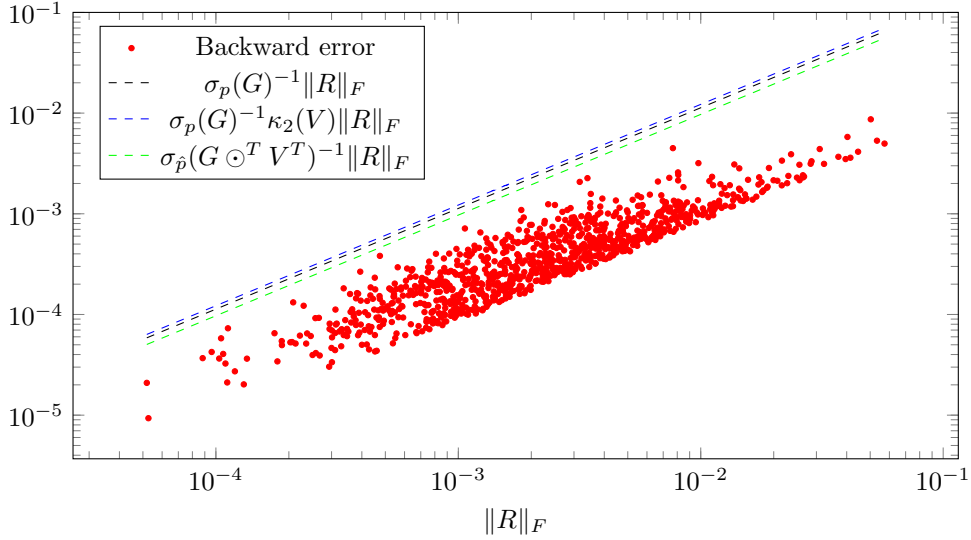

\subsubsection{The beam problem}
\label{subsubsec:beam unstructured}

We consider the delay eigenvalue problem obtained through the finite difference discretization of a one-dimensional beam with delayed stabilizing feedback, as described in \cite{VanBeeumen}:
\begin{equation}
\label{eq:beam}
    D(\lambda)= -\lambda I_n + A_0 + e^{-\lambda} A_1, \quad n=1000,
\end{equation}
where we have
\[
A_0 = \begin{bmatrix}
        A & -w^T \\
        -n\,w & n 
\end{bmatrix}, \quad A = \mbox{tridiag}(1,-2,1)\in \mathbb{R}^{\left(n-1\right) \times \left( n-1\right)}, \; w = \begin{bmatrix}
        0 & \ldots & 1
\end{bmatrix} \in \mathbb{R}^{1\times \left(n-1\right)}
\]
and $A_1= e_n e_n^T$ with $e_n$ the $n$-th vector of the canonical basis in $\mathbb{R}^n$. The coefficient matrices for this problem can be found in the example gallery presented in the NEP-PACK collection \cite{neppack}. In Figure \ref{fig:beam_figure}, we provide a comparison among the upper bounds for the unstructured backward error provided in Theorem \ref{th:back_err_with_eigenvectors} and Lemma \ref{lem:explicit_upper_bounds}. It is worth noting that the first two upper bounds for the case $p=3$ do not coincide, however this is not perceptible on the figure. Observe that for the case $p=10$, we may not use the first upper bound provided in Lemma \ref{lem:explicit_upper_bounds}.

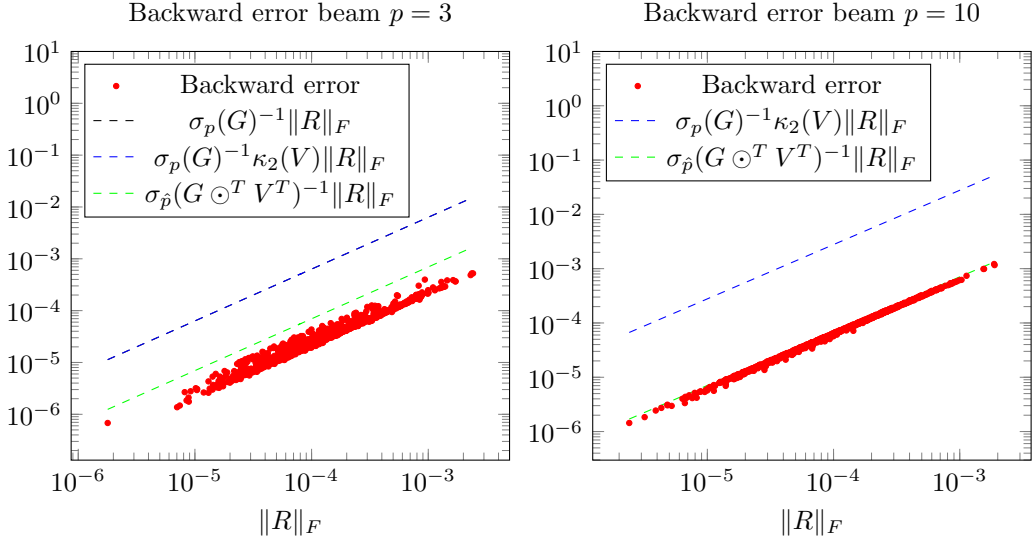
\begin{figure}[h]
     \centering
     \begin{tikzpicture}
       \begin{loglogaxis}[
             legend pos = north west,
             width=.5\linewidth,
             ymax=10,
             height= 70mm,
             title={Backward error beam $p=3$},
             xlabel={$\| R \|_F$}
                      ]
         \addplot[only marks, red, mark size = 1pt] table[y index = 1] {unstructured_bounds_check_beam.dat};
         \addplot[no marks, black, dashed] table[y index = 2] {unstructured_bounds_check_beam.dat};
        \addplot[no marks, blue, dashed] table[y index = 3] {unstructured_bounds_check_beam.dat};
         \addplot[no marks, green, dashed] table[y index = 4] {unstructured_bounds_check_beam.dat};
         \legend{Backward error, $\sigma_p(G)^{-1}\|R\|_F$, $\sigma_p(G)^{-1} \kappa_2(V)\| R\|_F$, $\sigma_{\hat p}(G \odot^T V^T)^{-1}\|R \|_F$}
       \end{loglogaxis}
     \end{tikzpicture}~\begin{tikzpicture}
       \begin{loglogaxis}[
             legend pos = north west,
             width=.5\linewidth,
             ymax=10,
             height=70mm,
             title={Backward error beam $p=10$},
             xlabel={$\| R \|_F$}
                      ]
         \addplot[only marks, red, mark size = 1pt] table[y index = 1] {unstructured_bounds_check_beam_p10.dat};
        \addplot[no marks, blue, dashed] table[y index = 3] {unstructured_bounds_check_beam_p10.dat};
         \addplot[no marks, green, dashed] table[y index = 4] {unstructured_bounds_check_beam_p10.dat};
         \legend{Backward error, $\sigma_p(G)^{-1} \kappa_2(V)\| R\|_F$, $\sigma_{\hat p}(G \odot^T V^T)^{-1}\|R \|_F$}
       \end{loglogaxis}
     \end{tikzpicture} 
      \caption{Comparison of the upper bounds for the unstructured backward error for the beam problem in \eqref{eq:beam}. On the left: we consider $p=3$ approximated eigenpairs. On the right: we consider $p=10$ approximated eigenpairs. }
     \label{fig:beam_figure}
 \end{figure}

\subsubsection{Test on randomly generated problems}
We now generate a set of random problems of the form 
\begin{equation} \label{eq:random-1}
  F(\lambda) = A_0 + \lambda A_1 + \lambda^2 I + e^{-\lambda} E_1 + e^{-2\lambda} E_2
\end{equation}
where $A_0,A_1,E_1,E_2$ are randomly symmetric generated matrices.
We fix a set of matrices $A_0,A_1,E_1,E_2 \in \mathbb{R}^{128 \times 128}$ of random matrices and a set of approximate eigenpairs $(\hat \lambda_i, \hat v_i)$ for $i=1,\ldots,p$. Then we generate $1000$ random perturbed matrix-valued function in the form:
\[
\tilde F(\lambda) = \tilde A_0 + \lambda \tilde A_1 + \lambda^2 \tilde I + e^{-\lambda} \tilde E_1 + e^{-2\lambda} \tilde E_2,
\]
and compute the backward error for the approximate eigenpairs $(\hat \lambda_i, \hat v_i)$ of the nonlinear eigenvalue problem associated with $\tilde F(\lambda)$. Then we test the error bounds for the unstructured backward error associated in Theorem~\ref{th:back_err_with_eigenvectors}, and compare it with the explicit bounds obtained in Lemma~\ref{lem:explicit_upper_bounds}. In Figure~\ref{fig:random_unstructured_plot}, the plot on the left provides the comparison for $p=3$, while the plot on the right a comparison of the bounds for $p=10$. Observe that if the number of considered approximate eigenpairs $p$ is strictly larger than the number of coefficients in the matrix-valued function, in Lemma \ref{lem:explicit_upper_bounds} the second bound does not hold.

 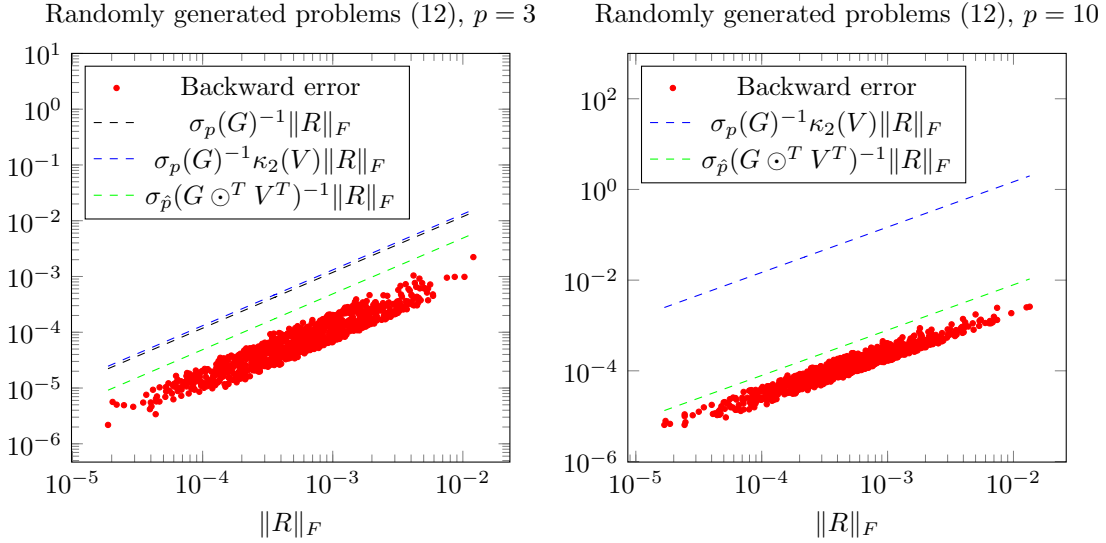
\begin{figure}[h]
 \centering
     \begin{tikzpicture}
       \begin{loglogaxis}[
             legend pos = north west,
             width=.5\linewidth,
             ymax=10,
             height=7cm,
             title={Randomly generated problems \eqref{eq:random-1}, $p=3$},
             xlabel={$\| R \|_F$}
         ]
         \addplot[only marks, red, mark size = 1pt] table[y index = 1] {unstructured_bounds_check.dat};
         \addplot[no marks, black, dashed] table[y index = 2] {unstructured_bounds_check.dat};
         \addplot[no marks, blue, dashed] table[y index = 3] {unstructured_bounds_check.dat};
         \addplot[no marks, green, dashed] table[y index = 4] {unstructured_bounds_check.dat};
         \legend{Backward error, $\sigma_p(G)^{-1}\|R\|_F$, $\sigma_p(G)^{-1} \kappa_2(V)\|R\|_F$, $\sigma_{\hat p}(G \odot^T V^T)^{-1}\|R\|_F$}
      \end{loglogaxis}
    \end{tikzpicture}~\begin{tikzpicture}
       \begin{loglogaxis}[
             legend pos = north west,
             width=.5\linewidth,
             ymax=1000,
             height=7cm,
             title={Randomly generated problems \eqref{eq:random-1}, $p=10$},
             xlabel={$\| R \|_F$}
         ]
         \addplot[only marks, red, mark size = 1pt] table[y index = 1] {unstructured_bounds_check2.dat};
         \addplot[no marks, blue, dashed] table[y index = 3] {unstructured_bounds_check2.dat};
         \addplot[no marks, green, dashed] table[y index = 4] {unstructured_bounds_check2.dat};
         \legend{Backward error, $\sigma_p(G)^{-1} \kappa_2(V)\|R\|_F$, $\sigma_{\hat p}(G \odot^T V^T)^{-1}\| R \|_F$}
       \end{loglogaxis}
     \end{tikzpicture}
     \caption{Comparison among the upper bounds for the unstructured backward error for the problem \eqref{eq:random-1}. On the left: we consider a set of $p=3$ approximate eigenpairs. On the right: we consider a set of $p=10$ approximate eigenpairs.}
     \label{fig:random_unstructured_plot}
 \end{figure}

\subsection{Structured case: linear subspaces}

\subsubsection{Randomly generated and sparse matrices}
\label{subsubsec:randomly_generated_sparse}

We consider again the matrix-valued function in \eqref{eq:random-1}, with randomly generated matrices of size $64 \times 64$, and we impose a sparsity pattern on the coefficients $A_0,A_1,E_1,E_2$, where we allow the sparsity patterns to be different from each others. We generate $1000$ random sets of coefficients for the matrix-valued function, where we preserve the sparsity pattern on the coefficients. We compute the structured backward error associated with a set of $p=3$ approximate eigenpairs for this set of randomly generated family of matrices, using the result in Theorem \ref{th:struct_back_err_with_eigenpairs}. In Figure \ref{fig:random_sparsity}, 
we provide test the upper bound for the structured backward error imposing sparsity patterns on the coefficients, against the exact formula for the structured backward error, as provided in Theorem \ref{th:struct_back_err_with_eigenpairs}. We report for completeness the upper bound for unstructured backward error provided by \ref{th:back_err_with_eigenvectors} (which does not hold in this case).

 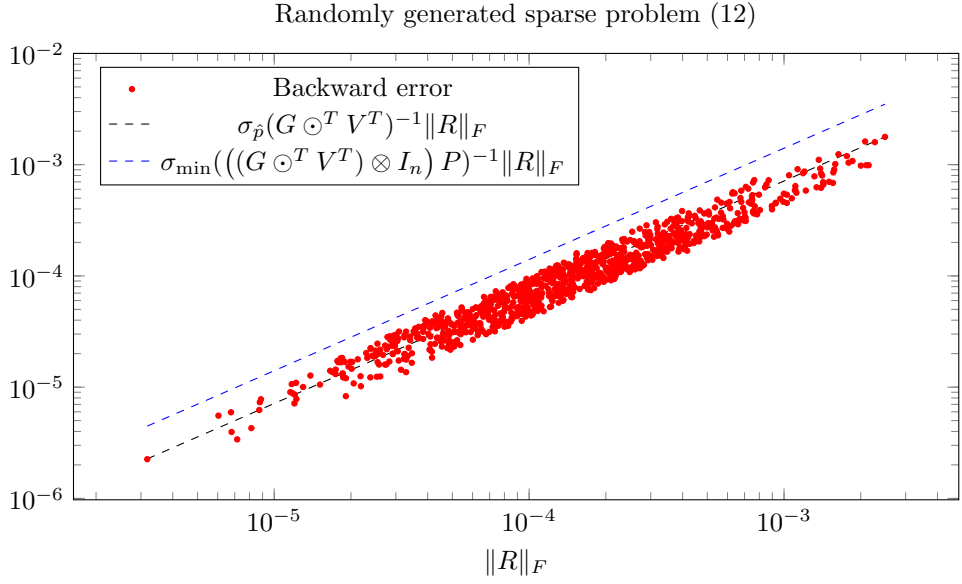
\begin{figure}[h]
     \centering
     \begin{tikzpicture}
       \begin{loglogaxis}[
             legend pos = north west,
             width=.9\linewidth,
             ymax=10^-2,
             height=75mm,
             title={Randomly generated sparse problem \eqref{eq:random-1}},
             xlabel={$\| R \|_F$}
         ]
         \addplot[only marks, red, mark size = 1pt] table[y index = 1] {linear_structured_bounds_check_fix_struct.dat};
         \addplot[no marks, black, dashed] table[y index = 2] {linear_structured_bounds_check_fix_struct.dat};
         \addplot[no marks, blue, dashed] table[y index = 3] {linear_structured_bounds_check_fix_struct.dat};
         \legend{Backward error, $\sigma_{\hat p}(G \odot^T V^T)^{-1} \| R \|_F$, $\sigma_{\min}(\left(
             (G \odot^T V^T) \otimes I_n \right) P)^{-1} \|{R}\|_F$}
       \end{loglogaxis}
     \end{tikzpicture}
      \caption{Test the upper bound for the structured backward error in Theorem \ref{th:struct_back_err_with_eigenpairs}, for the case of randomly generated sparse matrices in Subsection \ref{subsubsec:randomly_generated_sparse}. The bound for 
      unstructured case, which does not hold, is reported for completeness.}
     \label{fig:random_sparsity}
 \end{figure}

\subsubsection{Randomly generated symmetric matrices}

Consider again the nonlinear eigenvalue problem associated with \eqref{eq:random-1}, with randomly generated coefficients such that $A_i=A_i^T$ for $i=0,1$ and $E_j=E_j^T$ for $j=1,2$. As in the previous case, we run $1000$ tests, for a set of $p=3$ approximated eigenpairs. In Figure \ref{fig:random_symmetric_3bounds}, we consider an example of size $n=64$ and compute the structured backward error imposing the symmetry on the coefficient matrices, provided in Theorem \ref{thm:backward-error-symmetric}. We provide a comparison among the upper bound for general linear structures in Theorem \ref{th:struct_back_err_with_eigenpairs} and the one specialized for symmetry structures in Corollary \ref{cor:corollary_bound_symm}. 

Then we consider two randomly generated and symmetric problems as in \eqref{eq:random-1}, where the dimension of the coefficients is $n=128$ and $n=2048$. In Figure  \ref{fig:random_symmetric}, we test the upper bound in Corollary \ref{cor:corollary_bound_symm} against the structured backward error obtained using Theorem \ref{thm:backward-error-symmetric}, comparing it with the  one for unstructured backward error.

 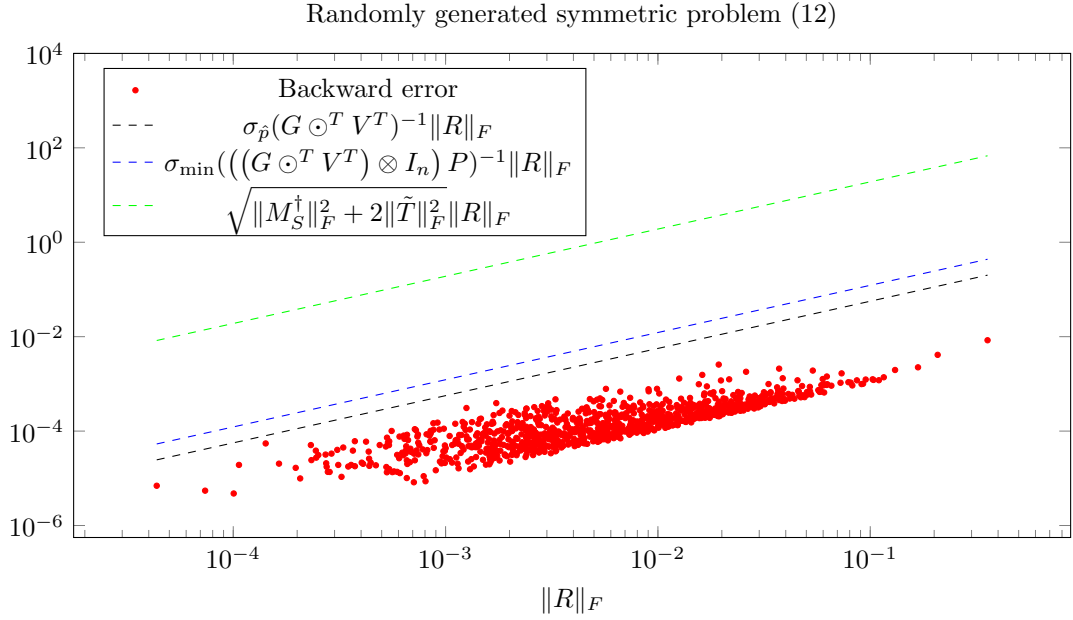
\begin{figure}[h]
     \centering
     \begin{tikzpicture}
       \begin{loglogaxis}[
             legend pos = north west,
             width=\linewidth,
             ymax=10^4,
             height=80mm,
             title={Randomly generated symmetric problem \eqref{eq:random-1}},
             xlabel={$\| R \|_F$}
         ]
         \addplot[only marks, red, mark size = 1pt] table[y index = 1] {linear_structured_bounds_check_symm_bound_n64.dat};
         \addplot[no marks, black, dashed] table[y index = 2] {linear_structured_bounds_check_symm_bound_n64.dat};
         \addplot[no marks, blue, dashed] table[y index = 3] {linear_structured_bounds_check_symm_bound_n64.dat};
         \addplot[no marks, green , dashed] table[y index = 4] {linear_structured_bounds_check_symm_bound_n64.dat};
         \legend{Backward error, $\sigma_{\hat p}(G \odot^T V^T)^{-1} \| R \|_F$,
          $\sigma_{\min}(\left( \left(
            G \odot^T V^T
           \right) \otimes I_n \right) P)^{-1} \|{R}\|_F$,
         $\sqrt{\| M_{S}^{\dagger}\|_F^2 + 2 \|\tilde T \|_F^2 }\| R\|_F$}
       \end{loglogaxis}
     \end{tikzpicture}
      \caption{Comparison between the bounds for structured backward error in Theorem \ref{th:struct_back_err_with_eigenpairs} and
 Corollary \ref{cor:corollary_bound_symm}, applied to problem \eqref{eq:random-1} with symmetric coefficients. For completeness, we report the unstructured bound.}
     \label{fig:random_symmetric_3bounds}
 \end{figure}

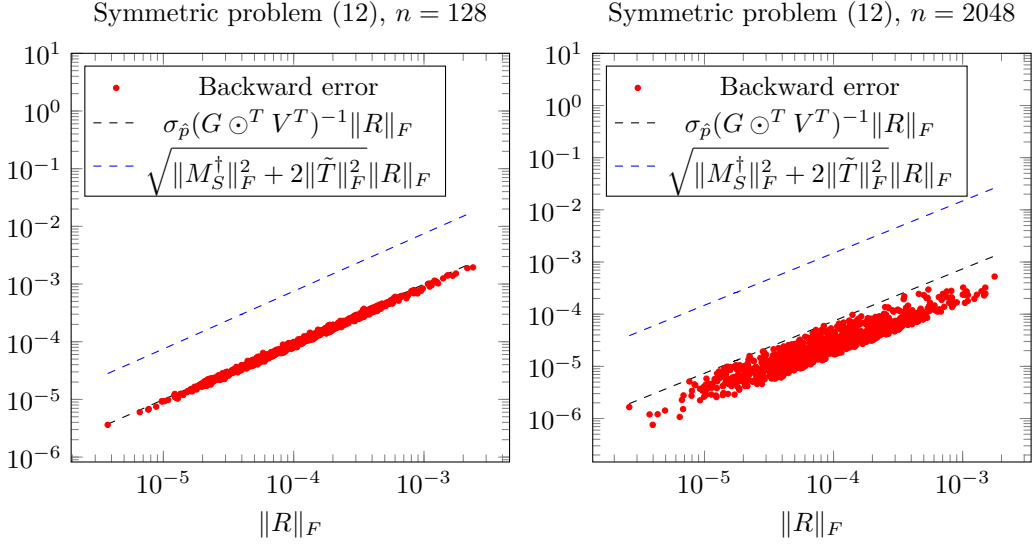
\begin{figure}[h!]
    \centering
    \begin{tikzpicture}
      \begin{loglogaxis}[
            legend pos = north west,
            width=.5\linewidth,
            ymax=10,
            height=7cm,
            title={Symmetric problem \eqref{eq:random-1}, $n=128$},
            xlabel={$\| R \|_F$}
        ]
        \addplot[only marks, red, mark size = 1pt] table[y index = 1] {symmetric_bounds_check_small.dat};
        \addplot[no marks, black, dashed] table[y index = 2] {symmetric_bounds_check_small.dat};
        \addplot[no marks, blue, dashed] table[y index = 3] {symmetric_bounds_check_small.dat};
        \legend{Backward error, $\sigma_{\hat p}(G \odot^T V^T)^{-1} \| R \|_F$, $\sqrt{\| M_{S}^{\dagger}\|_F^2 + 2 \|\tilde T \|_F^2 }\| R\|_F$}
      \end{loglogaxis}
    \end{tikzpicture}~\begin{tikzpicture}
      \begin{loglogaxis}[
            legend pos = north west,
            width=.5\linewidth,
            ymax=10,
            height=7cm,
            title={Symmetric problem \eqref{eq:random-1}, $n=2048$},
            xlabel={$\| R \|_F$}
        ]
    \addplot[only marks, red, mark size = 1pt] table[y index = 1] {symmetric_bounds_check_large.dat};
        \addplot[no marks, black, dashed] table[y index = 2] {symmetric_bounds_check_large.dat};
        \addplot[no marks, blue, dashed] table[y index = 3] {symmetric_bounds_check_large.dat};
        \legend{Backward error, $\sigma_{\hat p}(G \odot^T V^T)^{-1} \| R \|_F$, $\sqrt{\| M_{S}^{\dagger}\|_F^2 + 2 \|\tilde T \|_F^2 }\| R\|_F$}
      \end{loglogaxis}
    \end{tikzpicture}
    \caption{Test the structured bound in Corollary \ref{cor:corollary_bound_symm}, for randomly generated symmetric problems in \eqref{eq:random-1}. On the left: size $n=128$. On the right: size $n=2048$. The bound for unstructured case, which does not hold, is reported for completeness.}
    \label{fig:random_symmetric}
\end{figure}

\subsection{Riemannian optimization}
\label{sub:experiments_riemannian}

\subsubsection{Quadratic polynomial eigenvalue problem}

We consider the nonlinear matrix-valued function:
\begin{equation}
    \label{eq:quadratic_eigenvalue_problem}
F(\lambda) = A_0 + \lambda A_1 + \lambda^2 A_2 \in \mathbb{R}^{n \times n}, \quad n=10000,
\end{equation}
where the matrix $A_0= \mbox{tridiag}(1,-2,1)$, $A_1=-UU^T$ is a low-rank matrix with a randomly generated matrix $U \in \mathbb{R}^{10000 \times 2}$ and $A_2$ is the identity matrix. We consider an approximation of two eigenpairs $(\hat \lambda_i, \hat v_i)$ for $i=1,2$ and perturb the matrix coefficients keeping the same structures:
\[
\widetilde{F}(\lambda)= (A_0 + \tilde A_0)+ \lambda \tilde A_1 + \lambda^2 (A_2 + \tilde A_2),
\]
where $\tilde A_0$ is a randomly generated tridiagonal matrix, $\tilde A_1 = -(U+\tilde U)(U +\tilde U)^T$ with $\tilde U \in \mathbb{R}^{n \times 2}$ randomly generated and $\tilde A_2$ a multiple of the identity. The norm of the perturbation $\| \tilde F- F \|_F$ is in the order of $1.992017$.

In order to apply the method proposed in Subsection \ref{sec:nonlinear-structures}, we consider the following product manifold:
\[
\mathcal{S}:=\mathcal{S}_1  \times \mathbb{R}_2^{n\times n} \times \mathcal{S}_2,
\]
where $\mathcal{S}_1$ is the manifold of sparse matrices with the same sparsity patterns of $A_0$, $\mathbb{R}_2^{n\times n}$ is the manifold of rank $2$ real matrices of size $n \times n$ and $\mathcal{S}_2$ is the manifold of the matrices that are multiples of the identity. The implementation of the method requires the use of the \texttt{manopt} package for MATLAB, version $7.1$. Observe that the manifold $\mathcal{S}_2$ is not available in \texttt{manopt}, then we used our implementation of this manifold.

The running time for the computation of the structured backward error associated with two approximate eigenpairs of $F(\lambda)$ is $199.1613$ seconds. We obtain an upper bound for the backward error equal to $3.521105 \times 10^{-2}$ and a norm of the residual equal to $4.015467 \times 10^{-9}$. 

In practice, this experiment can be repeated as coded in Listing \ref{lst:riemann}, where \texttt{f} is a function for the evaluation of $[1, \lambda, \lambda^2]$ and \texttt{(V,L)} are approximate eigenpairs. The command \texttt{be\_riemannian} calls the manifolds that we need for the optimization procedure, where \texttt{'identity'} refers to our implementation of the manifold $\mathcal{S}_2$. The MATLAB functions can be found in the github repository \url{https://github.com/miryamgnazzo/backward-error-nonlinear}.

\begin{lstlisting}[label={lst:riemann}, caption={Code for the experiment in \eqref{eq:quadratic_eigenvalue_problem}},
frame=single,
style=Matlab-Pyglike]
F = { A0, A1, A2 }; %cell array of coefficient matrices

D = be_riemannian(F, @f, ...
        { 'sparse', 'low-rank', 'identity' }, V, L);
nrm = be_norm(D); %Computed backward error;
\end{lstlisting}

\subsubsection{The beam problem with prescribed sparsity pattern}

Consider again the beam problem stated in Subsection \ref{subsubsec:beam unstructured}. We consider different dimensions $n$ for the matrix-valued function in \eqref{eq:beam} and we apply the Riemannian optimization-based approach in Section \ref{sec:nonlinear-structures}, preserving the sparsity structures of the coefficients, that leads to the product manifold:
\[
\mathcal{S} := \mathcal{S}_0 \times \mathcal{S}_1 \times \mathcal{S}_2, 
\]
where $\mathcal{S}_0$ is the manifold of matrices that are multiples of the identity, $\mathcal{S}_1$ is the manifold of tridiagonal matrices and $\mathcal{S}_2$ the one of multiples of the matrix $e_n e_n^T$. Observe that the involved structures are linear, nevertheless we compute an approximation of the structured backward error, in order to provide a few examples on matrices of large size. 

We compute $\hat V \in \mathbb{R}^{n \times 3}$ and $\hat \Lambda \in \mathbb{R}^{3 \times 3}$, approximations of $p=3$ eigenvectors and eigenvalues of $D(\lambda)$ in \eqref{eq:beam}, respectively, then perturb it by 
\[
\Delta D(\lambda) = -\Delta_0 \lambda + \Delta_1 + \Delta_2 e^{-\lambda},
\]
where $F_i + \Delta_i \in \mathcal{S}_i$ for $i=0,1,2$, since $F_0=I_n$, $F_1=A_0$ and $F_2=A_1$. The algorithm in Subsection \ref{sec:nonlinear-structures} provides final matrices $\Delta F_1, \Delta  F_2, \Delta F_3$, which we use to define the (approximated) structured backward error $\eta_{\mathcal{S}}= \|\begin{bmatrix} \Delta F_0 & \Delta F_1 & \Delta  F_2 \end{bmatrix} \|_F$. We test the accuracy of our solution computing the norm of the residual
\[
R:= -(I_n + \Delta F_0) \hat V \hat \Lambda + (A_0 + \Delta F_1) \hat V + (A_1 + \Delta F_2) \hat V \exp(-\hat \Lambda). 
\]
In Table \ref{tab:comparison_beam}, we collect the results obtained considering different sizes $n$. In particular, we provide a comparison among the elapsed time (expressed in seconds), the (approximated) structured backward error $\eta_{\mathcal{S}}$, the norm of the residual $R$ and the Frobenius norm of the starting perturbation matrices.

\begin{table}[h!]
    \centering
    \begin{tabular}{|c|c|c|c|c|}
    \hline
        $n$ & time (in seconds) & $\eta_{\mathcal{S}}$ & $\| R \|_F$ & $\| \Delta D\|_F$ \\
        \hline
        \hline
        $10^3$ & $52.9975$ & $5.564350 \times 10^{-3}$ & $2.351591\times 10^{-9}$ & $5.625286 \times 10^{-3}$ \\
        $2 \times 10^3$ & $62.2178$ & $7.798143\times 10^{-3}$ & $9.967606\times 10^{-10}$ & $7.811920\times 10^{-3}$ \\
        $5 \times 10^3$ & $94.1982$ & $1.383995 \times 10^{-2}$ & $3.560193\times 10^{-9}$ & $1.421781 \times 10^{-2}$ \\
        $10^4$ & $145.4670$ & $1.814495 \times 10^{-2}$ & $4.786304 \times 10^{-9}$ & $1.859965 \times 10^{-2}$ \\
        $2 \times 10^4$ & $249.8171$ & $2.584339 \times 10^{-2}$ & $2.480937 \times 10^{-9}$ & $2.616653 \times 10^{-2}$\\ 
        $5 \times 10^4$ & $945.6344$ & $3.759764 \times 10^{-2}$ & $6.855576 \times 10^{-9}$ & $3.877037 \times 10^{-2}$\\
        $10^5$ & $1.6083 \times 10^{3}$ & $5.295991 \times 10^{-2}$ & $1.012763 \times 10^{-8}$ & $5.703946 \times 10^{-2}$\\
        \hline
    \end{tabular}
    \caption{Results for the beam problem \eqref{eq:beam}, with different sizes $n$.}
    \label{tab:comparison_beam}
\end{table}

\section*{Conclusions}

We propose a backward error analysis for nonlinear eigenvalue problems
given in split form. We presented a novel formula for the computation of 
the backward errors for a given set of eigenpairs or eigenvalues, and 
explicitly, computable, and inexpensive upper bounds for them. These 
bounds have been verified to be tight and descriptive on a set of 
examples arising from standard benchmark collections. 

We discussed in detail how to impose different structures on the backward
errors. For the case of coefficients living in a linear subspace, we have 
extended the previous analysis, and provided computable bounds. The bounds 
are in particular still relatively inexpensive for the relevant case of 
symmetric coefficients. 

For more general structures, where coefficients are in a
differentiable manifold, we have provided an effective 
algorithm for the computation of the backward error, based on a 
Riemannian optimization technique. This allows to compute backward errors for 
problems with low-rank coefficients, but also for the ones where the constraint is linear, such as prescribed sparsity pattern, or symmetries, and any 
combination of these. We have verified the effectiveness and 
the scalability of this approach, which is able to give explicit bounds 
for large-scale structured nonlinear eigenvalue problems. 

\section*{Acknowledgments}
Miryam Gnazzo and Leonardo Robol are members of the INdAM Research group GNCS (Gruppo Nazionale di Calcolo Scientifico). The work of Leonardo Robol was partially supported by the National Research Center in High Performance Computing, Big Data and Quantum Computing (CN1 -- Spoke 6), by the MIUR Excellence Department Project awarded to the Department of Mathematics, University of Pisa (CUP I57G22000700001), 
and by the Italian Ministry of University and Research (MUR) through the PRIN 2022 ``MOLE: Manifold constrained Optimization and LEarning'',  code: 2022ZK5ME7 MUR D.D. financing decree n. 20428 of November 6th, 2024 (CUP I53C24002260006).
The work of Miryam Gnazzo was partially supported 
by the PRIN 2022 Project ``Low-rank Structures and Numerical Methods in Matrix and Tensor Computations and their Application''.
Furthermore, during part of the preparation of this work, Miryam Gnazzo was affiliated with Gran Sasso Science Institute, L'Aquila (Italy).

\section*{Declarations}

\textbf{Conflict of interest} The authors have no competing interests to declare that are relevant to the content of
this article.

\bibliographystyle{plain}
\bibliography{biblio}

\begin{thebibliography}{10}

\bibitem{AbsilTrustregion}
P.~A. Absil, C.~G. Baker, and K.~A. Gallivan.
\newblock Trust-region methods on {R}iemannian manifolds.
\newblock {\em Found. Comput. Math.}, 7(3):303–330, jul 2007.

\bibitem{Alam}
Bibhas Adhikari, Rafikul Alam, and Daniel Kressner.
\newblock Structured eigenvalue condition numbers and linearizations for matrix
  polynomials.
\newblock {\em Linear Algebra and its Applications}, 435(9):2193--2221, 2011.

\bibitem{AhmMehr}
Sk.~Safique Ahmad and Volker Mehrmann.
\newblock Backward errors and pseudospectra for structured nonlinear eigenvalue
  problems.
\newblock {\em Operators and Matrices}, 10:539--556, 09 2016.

\bibitem{Bertsekas}
D.P. Bertsekas.
\newblock {\em Nonlinear Programming}.
\newblock Athena Scientific, Belmont, 1999.

\bibitem{NLEVP}
Timo Betcke, Nicholas~J. Higham, Volker Mehrmann, Christian Schr\"{o}der, and
  Fran\c{c}oise Tisseur.
\newblock Nlevp: A collection of nonlinear eigenvalue problems.
\newblock {\em ACM Trans. Math. Softw.}, 39(2), feb 2013.

\bibitem{Boumal}
Nicolas Boumal.
\newblock {\em An Introduction to Optimization on Smooth Manifolds}.
\newblock Cambridge University Press, Cambridge, 2023.

\bibitem{BoumalMishraAbsil}
Nicolas Boumal, Bamdev Mishra, P.-A. Absil, and Rodolphe Sepulchre.
\newblock Manopt, a {M}atlab toolbox for optimization on manifolds.
\newblock {\em Journal of Machine Learning Research}, 15(42):1455--1459, 2014.

\bibitem{Corless}
Robert~M. Corless.
\newblock Pseudospectra of exponential matrix polynomials.
\newblock {\em Theoretical Computer Science}, 479:70--80, 2013.

\bibitem{DopicoKron}
Froilan Dopico, Piers Lawrence, Javier Pérez, and Paul Van~Dooren.
\newblock Block {K}ronecker linearizations of matrix polynomials and their
  backward errors.
\newblock {\em Numerische Mathematik}, 140, 10 2018.

\bibitem{DopicoPomes}
Froilan Dopico and Kenet Pomés~Portal.
\newblock Structured eigenvalue condition numbers for parameterized
  quasiseparable matrices.
\newblock {\em Numerische Mathematik}, 134, 11 2016.

\bibitem{Feppon}
Florian Feppon and Pierre F.~J. Lermusiaux.
\newblock A geometric approach to dynamical model order reduction.
\newblock {\em SIAM Journal on Matrix Analysis and Applications},
  39(1):510--538, 2018.

\bibitem{GutTiss}
Stefan Güttel and Françoise Tisseur.
\newblock The nonlinear eigenvalue problem.
\newblock {\em Acta Numerica}, 26:1–94, 2017.

\bibitem{Hadeler}
K.~P. Hadeler.
\newblock Mehrparametrige und nichtlineare eigenwertaufgaben.
\newblock {\em Archive for Rational Mechanics and Analysis}, 27(4):306–328,
  1967.

\bibitem{Jacobsson}
Simon Jacobsson, Raf Vandebril, Joeri~Van der Veken, and Nick Vannieuwenhoven.
\newblock Approximating maps into manifolds with lower curvature bounds.
\newblock {\em ArXiv}, abs/2403.16785, 2024.

\bibitem{neppack}
Elias Jarlebring, Max Bennedich, Giampaolo Mele, Emil Ringh, and Parikshit
  Upadhyaya.
\newblock Nep-pack: A julia package for nonlinear eigenproblems - v0.2, 2018.

\bibitem{KarKreMeng}
Michael Karow, Daniel Kressner, and Emre Mengi.
\newblock Nonlinear eigenvalue problems with specified eigenvalues.
\newblock {\em SIAM Journal on Matrix Analysis and Applications},
  35(3):819--834, 2014.

\bibitem{LietaertMeer}
Pieter Lietaert, Karl Meerbergen, Javier Pérez, and Bart Vandereycken.
\newblock {Automatic rational approximation and linearization of nonlinear
  eigenvalue problems}.
\newblock {\em IMA Journal of Numerical Analysis}, 42(2):1087--1115, 02 2021.

\bibitem{LiuBoumal}
Changshuo Liu and Nicolas Boumal.
\newblock Simple algorithms for optimization on {R}iemannian manifolds with
  constraints.
\newblock {\em Appl. Math. Optim.}, 82(3):949–981, dec 2020.

\bibitem{Sharma}
Anshul Prajapati and Punit Sharma.
\newblock Structured eigenvalue backward errors for rational matrix functions
  with symmetry structures.
\newblock {\em BIT Numerical Mathematics}, 64, 02 2024.

\bibitem{Sun}
Ji-guang Sun.
\newblock Backward perturbation analysis of certain characteristic subspaces.
\newblock {\em Numerische Mathematik}, 65(1):357–382, December 1993.

\bibitem{TisseurChart}
Fran\c{c}oise Tisseur.
\newblock A chart of backward errors for singly and doubly structured
  eigenvalue problems.
\newblock {\em SIAM Journal on Matrix Analysis and Applications},
  24(3):877--897, 2003.

\bibitem{Tisseur}
Françoise Tisseur.
\newblock Backward error and condition of polynomial eigenvalue problems.
\newblock {\em Linear Algebra and its Applications}, 309(1):339--361, 2000.

\bibitem{TissMeer}
Françoise Tisseur and Karl Meerbergen.
\newblock The quadratic eigenvalue problem.
\newblock {\em SIAM Review}, 43(2):235--286, 2001.

\bibitem{VanBeeumen}
Roel Van~Beeumen, Elias Jarlebring, and Wim Michiels.
\newblock A rank-exploiting infinite {A}rnoldi algorithm for nonlinear
  eigenvalue problems.
\newblock {\em Numerical Linear Algebra with Applications}, 23(4):607--628,
  2016.

\bibitem{van2000ubiquitous}
Charles~F Van~Loan.
\newblock The ubiquitous {K}ronecker product.
\newblock {\em Journal of computational and applied mathematics},
  123(1-2):85--100, 2000.

\end{thebibliography}

\end{document}